\documentclass[a4paper,oneside,11pt]{amsart} 
\usepackage[ansinew]{inputenc} 
\usepackage[T1]{fontenc}
\usepackage[english]{babel} 

\usepackage{enumerate} 
\usepackage{booktabs}

\usepackage{amsmath} 
\usepackage{amsfonts} 
\usepackage{amssymb}
\usepackage{amsthm}
\usepackage{latexsym}
\usepackage{bbm}
\usepackage{setspace}
\usepackage[top=3cm , bottom=3cm , left=2.1cm , right=2.1cm]{geometry}
\usepackage[pdftex]{graphicx}
  \DeclareGraphicsExtensions{.pdf,.png,.jpg,.jpeg,.mps}
  \usepackage{pgf}
  \usepackage{tikz}
\usepackage{floatflt}
\usepackage{url}
\usepackage{mathrsfs}
\usepackage{esint}
\usepackage{comment}

\makeatletter

\@namedef{subjclassname@2010}{%

  \textup{2010} Mathematics Subject Classification}

\makeatother

\theoremstyle{plain}
\newtheorem{lemma}{Lemma}
\newtheorem{theorem}{Theorem} 
\newtheorem{prop}{Proposition}
\newtheorem*{prop*}{Proposition}
\newtheorem*{lemma*}{Lemma}
\newtheorem{cor}{Corollary}
\newtheorem*{theorem*}{Theorem}

\theoremstyle{definition}
\newtheorem*{def*}{Definition}

\theoremstyle{remark}

\newtheorem*{remark}{Remark}
\newtheorem*{acks}{Acknowledgements}

\newcommand{\R}{\mathbb{R}}

\newcommand{\la}{\langle}
\newcommand{\ra}{\rangle}
\newcommand{\dom}{\textup{\textsf{Dom}}}
\newcommand{\real}{\textup{Re}}
\newcommand{\imag}{\textup{Im}}
\newcommand{\arc}{\textup{arc}}
\newcommand{\ray}{\textup{ray}}

\title[Wave extension problem]{Wave extension problem for the fractional Laplacian}
\author{Mikko Kemppainen}
\address[M. Kemppainen]{Departamento de Matem\'{a}ticas, Universidad Aut\'{o}noma de Madrid, 28049 Madrid, Spain}
\curraddr{Department of Mathematics and Statistics, University of Helsinki, FI-00014 Helsinki, Finland}
\email{mikko.k.kemppainen@helsinki.fi}

\author{Peter Sj\"{o}gren}
\address[P. Sj\"{o}gren]{Mathematical Sciences, University of Gothenburg and Mathematical Sciences, Chalmers, SE-412 96 G\"{o}teborg, Sweden}
\email{peters@chalmers.se}

\author{Jos\'{e} Luis Torrea}
\address[J. L. Torrea]{Departamento de Matem\'{a}ticas, Universidad Aut\'{o}noma de Madrid and ICMAT, 28049 Madrid, Spain}
\email{joseluis.torrea@uam.es}

\begin{document}

%

\begin{abstract}
  We show that the fractional Laplacian can be viewed as a Dirichlet-to-Neumann map for a degenerate hyperbolic problem,
  namely, the wave equation with an additional diffusion term that blows up at time zero. A solution to this wave extension
  problem is obtained from the Schr\"{o}dinger group by means of an oscillatory subordination formula, which also allows us
  to find kernel representations for such solutions. Asymptotics of related oscillatory
  integrals are analysed in order to determine the correct domains for initial data in the 
  general extension problem involving non-negative self-adjoint operators. An alternative approach using Bessel functions
  is also described.
\end{abstract}

\maketitle

\section{Introduction}

In the last decade a lot of attention has been devoted to fractional powers of the Laplacian $-\Delta$. 
The seminal paper by L. Caffarelli and L. Silvestre \cite{CAFFARELLISILVESTRE} showed that the operator
$(-\Delta)^{\sigma}$, with $0 < \sigma < 1$, 
mediates between Dirichlet and Neumann boundary values on $\R^d$ for a certain degenerate 
elliptic problem in the upper half-space $\R^{d+1}_+$.

Motivated by their work we show that up to a multiplicative constant
the fractional Laplacian can also be viewed as a Dirichlet-to-Neumann map 
$f\mapsto \lim_{t\to 0} t^{1-2\sigma} \partial_t u$ for the hyperbolic problem
\begin{equation}
\label{introWE}
  \begin{cases}
    \partial_t^2 u + \frac{1-2\sigma}{t}\partial_t u = \Delta u, \\
    u(\cdot , 0) = f .
  \end{cases}
\end{equation}
Here  $0 < \sigma < 1$, and for  $\sigma = \frac{1}{2}$ we have the classical wave equation.

We use the language of semigroups to study this problem in a parallel way to that used
by Stinga and Torrea \cite{STINGATORREA} for discussing the elliptic problem of Caffarelli and Silvestre
in the generality of non-negative self-adjoint operators. 
For an extension of their method to more general classes of
semigroup generators, see \cite{STINGAWAVE}.
We also wish to point out that an alternative approach to
problem \eqref{introWE} by means of Bessel functions is available (see Section \ref{sectionBessel}).

Our main result says that a solution to \eqref{introWE} is given by means of the `oscillatory' subordination formula
\begin{equation}
\label{introsubord}
  u(\cdot , t) = \frac{i^\sigma t^{2\sigma}}{4^\sigma \Gamma (\sigma)} \int_0^\infty e^{-i\frac{t^2}{4s}} \, e^{is\Delta}f 
    \,\frac{ds}{s^{1+\sigma}} ,
\end{equation}
also when $\Delta$ is replaced by a more general non-negative self-adjoint operator. 
It allows us to make use of the Schr\"{o}dinger kernel 
in order to find integral representations for solutions to the associated Neumann problem in different dimensions.

It is perhaps interesting to note that in the classical case when $\sigma = \frac{1}{2}$, the formula \eqref{introsubord}
leads to the identity
\begin{equation}
\label{wavegroup}
  \frac{e^{-it\sqrt{-\Delta}}}{\sqrt{-\Delta}}
  = \sqrt{\frac{i}{\pi}} \int_0^\infty e^{-i\frac{t^2}{4s}} \, e^{is\Delta} \,\frac{ds}{\sqrt{s}}
\end{equation}
for the wave group $(e^{-it\sqrt{-\Delta}})_{t\in\R}$. Recall in particular that the solution $u$
to the wave equation with Neumann initial data $g$ (given by the imaginary part of the wave group)
\begin{equation}
\label{sinsol}
  u(\cdot , t) = \frac{\sin (t\sqrt{-\Delta})}{\sqrt{-\Delta}} g
\end{equation}
can be expressed in dimensions $3$ and $2$, respectively, by 
\begin{equation}
\label{sphericalaverage}
  u(x,t) = t \fint_{\partial B(x,t)} g(y) \,dS(y) = \frac{c}{t} S|_{\partial B(0,t)} \ast g(x)
\end{equation}
and
\begin{equation}
\label{spatialaverage}
  u(x,t) = \frac{t^2}{2} \fint_{B(x,t)} \frac{g(y)}{(t^2 - |x-y|^2)^{\frac{1}{2}}} \,dy 
  = \frac{c' 1_{B(0,t)}}{(t^2 - |\cdot |^2)^{\frac{1}{2}}} \ast g(x) .
\end{equation}
(see \cite[Section 2.4, Equations (22) and (27)]{EVANS}).\footnote{The positive constants $c$, $c'$ appearing here and in what follows can have different values at different instances.}
These formulas highlight Huygens' principle concerning the finite propagation speed of solutions to the wave equation.
It would be peculiar if the smooth and fully supported Schr\"{o}dinger kernel would transform
by \eqref{wavegroup} into a singular and spherically supported one such as $\frac{1}{t} S|_{\partial B(0,t)}$
in \eqref{sphericalaverage}, and indeed there are restrictions in the relation of $\sigma$ and the dimension in our kernel
representations.

We will now describe the content of this paper. Let $L$ be a non-negative self-adjoint operator on
a Lebesgue space $\mathcal{L}^2$ and 
consider the Schr\"{o}dinger group $(e^{-isL})_{s\in\R}$.
In Theorem \ref{mainresult1} we show that for $0 < \sigma < 1$ the oscillatory integral
\begin{equation*}
  U_t^\sigma (f) = \frac{i^\sigma t^{2\sigma}}{4^\sigma \Gamma (\sigma)} \int_0^\infty e^{-i\frac{t^2}{4s}} \, e^{-isL}f \frac{ds}{s^{1+\sigma}},
\end{equation*}
when interpreted as an improper integral, converges weakly in $\mathcal{L}^2$ for suitable $f$ and solves the equation
\begin{equation}
  \label{introWEgen}
  \partial_t^2 u + \frac{1-2\sigma}{t} \partial_t u = -Lu, \quad t>0, 
\end{equation}
with the initial data
\begin{equation*}
  u(\cdot , 0) = f \quad \textup{and} \quad 
  \partial_t^\sigma u(\cdot , 0) = i^{-2\sigma} 
  c_\sigma L^\sigma f .
\end{equation*}
where 
\begin{equation*}
  \partial_t^\sigma u(\cdot , t) = \frac{1}{2\sigma} t^{1-2\sigma} \partial_t u(\cdot , t) .
\end{equation*}
Theorem \ref{mainresult1} is established via the Spectral theorem (see Proposition \ref{lambdaeq}); the requirement that
$f\in\dom (L^{\frac{\sigma}{2} + \frac{3}{4}})$ is needed to deal with the asymptotics of oscillatory integrals in Lemmas
\ref{firstestimate} and \ref{secondestimate}. Our method for analysing these integrals and overcoming the delicate
problems of convergence involves `sidestepping' the imaginary half-axis and moving to more suitable complex 
paths that provide sufficient decay.
In Corollary \ref{corollary}, a solution $u$ to \eqref{introWEgen} with real initial data
$u(\cdot , 0) = f$ and $\partial_t^\sigma u(\cdot , 0) = g$ is presented in terms of real and imaginary parts of
$U_t^\sigma$ as
\begin{equation*}
  u(\cdot , t) = 
  c_\sigma \real (i^{1-2\sigma} U_t^\sigma (f))
  - 
  c_\sigma' \imag (U_t^\sigma (L^{-\sigma}g)) .
\end{equation*}
For $\sigma = \frac{1}{2}$ this reduces to
\begin{equation}
\label{classicalsol}
  u(\cdot , t) = \real (e^{-it\sqrt{L}}f) - \imag \Big( \frac{e^{-it\sqrt{L}}}{\sqrt{L}} g \Big)
  = \cos (t\sqrt{L})f + \frac{\sin(t\sqrt{L})}{\sqrt{L}}g .
\end{equation}

In Section \ref{sectionkernel}, we return to the case of $L=-\Delta$ and 
find kernel representations for solutions $u$ to the Neumann problem
\begin{equation*}
  \begin{cases}
    \partial_t^2 u + \frac{1-2\sigma}{t}\partial_t u = \Delta u , \\
    u(\cdot , t) = 0 , \;\, \partial_t^\sigma u(\cdot , t) = g
  \end{cases}
\end{equation*}
on $\R^d$ for $d=1,2,\ldots , 5$. From the subordination formula \eqref{introsubord} we obtain
\begin{equation*}
  U_t^\sigma ((-\Delta)^{-\sigma}g) = 
  - \frac{i^{\sigma}}{\Gamma(\sigma)} \int_0^\infty e^{-i\frac{t^2}{4s}} \, e^{is\Delta}g \frac{ds}{s^{1-\sigma}},
\end{equation*}
and inserting the Schr\"{o}dinger kernel
\begin{equation*}
  \frac{1}{(4\pi is)^{\frac{d}{2}}} \, e^{i\frac{|x|^2}{4s}},
\end{equation*}
we arrive at the expression
\begin{equation}
\label{introexpression}
  c_{d,\sigma} i^{\sigma - \frac{d}{2}}
  \int_0^\infty e^{-i\frac{t^2}{4s}} \int_{\R^d} e^{i\frac{|x-y|^2}{4s}} g(y) dy \frac{ds}{s^{1+\frac{d}{2} - \sigma}} ,
\end{equation}
the imaginary part of which gives the solution $u$.

Theorem \ref{kernel123} concerns dimensions $d=1,2,3$ and $\sigma$ for which $0 < \frac{d}{2} - \sigma < 1$ and derives
from \eqref{introexpression} the formula
\begin{equation*}
  u(x,t) = c_{d,\sigma} \int_{B(x,t)} \frac{g(y)}{(t^2-|x-y|^2)^{\frac{d}{2} - \sigma}} dy .
\end{equation*}
The dimensions $d=3,4,5$ are considered in Theorem \ref{kernel345}, which states that for $\sigma$ such that
$1 < \frac{d}{2} - \sigma < 2$ we have
\begin{equation*}
  u(x,t) = \frac{c_{d,\sigma}}{t^2} 
    \int_{B(x,t)} \frac{2(\sigma + 1)g(y) + (y-x)\cdot \nabla g(y)}{(t^2 - |x-y|^2)^{\frac{d}{2} - \sigma - 1}} dy .
\end{equation*}
The limiting cases when $\frac{d}{2} - \sigma$ tends to either zero or one are studied in Theorem \ref{limitingcases};
in dimensions $d=2,3,4$ we then have
\begin{equation*}
    u_\sigma (x,t) \longrightarrow \frac{1}{c_dt} \int_{\partial B(x,t)} g(y)dS(y) .
\end{equation*}

In Theorem \ref{Besselsol} we show that the solution to
\begin{equation*}
  \begin{cases}
    \partial_t^2 u + \frac{1-2\sigma}{t}\partial_t u = \Delta u , \\
    u(\cdot , t) = f , \;\, \partial_t^\sigma u(\cdot , t) = g
  \end{cases}
\end{equation*}
with Schwartz initial data $f$ and $g$ is unique and can be given in terms of Bessel functions $J_{\pm\sigma}$, namely,
\begin{equation*}
  u(\cdot , t) = c_\sigma (t\sqrt{-\Delta})^\sigma J_{-\sigma}(t\sqrt{-\Delta})f + c_\sigma' (t\sqrt{-\Delta})^\sigma
  J_\sigma(t\sqrt{-\Delta}) (-\Delta)^{-\sigma}g .
\end{equation*}
For $\sigma = \frac{1}{2}$ this coincides with the classical formula \eqref{classicalsol} (with $L=-\Delta$);
indeed
\begin{equation*}
  \sqrt{r} J_{-1/2}(r) = \sqrt{\frac{2}{\pi}} \cos r \quad \textup{and} \quad 
  \sqrt{r} J_{1/2}(r) = \sqrt{\frac{2}{\pi}} \sin r
\end{equation*}
(and $c_{1/2} = c_{1/2}' = \sqrt{\pi / 2}$).
Growth estimates for Bessel functions allow us to deduce, by means of the Fourier transform, 
fixed-time estimates for solutions of the equation (see Theorem \ref{fixedtime}). Finally, a classical integral representation
for (modified) Bessel functions $K_\sigma$ is converted into an oscillatory integral formula coinciding with \eqref{introsubord},
thus closing the circle.

In order to avoid ambiguity, let us agree that $i^\alpha = e^{i\pi\alpha / 2}$. Moreover, by saying that an 
integral $\int_0^\infty$ is convergent, we mean that the limit of $\int_\varepsilon^R$ exists as $\varepsilon\to 0$
and $R\to\infty$.

By $\alpha \lesssim \beta$ we mean that there exists a constant $C$ such that $\alpha \leq C \beta$. Two quantities
$\alpha$ and $\beta$ are comparable, $\alpha \sim \beta$, if $\alpha\lesssim\beta$ and $\beta\lesssim\alpha$.

\begin{acks}
The first author gratefully acknowledges the financial support from the Finnish Academy of Science and Letters, V{\"a}is{\"a}l{\"a} Foundation, and from the Finnish Centre of Excellence in Analysis and Dynamics Research. He is thankful for the hospitality of the Department of Mathematics at the Autonomous University of Madrid during his stay.
The third author is supported by the grant MTM2011-28149-C02-01 from Spanish Government.
\end{acks}

\section{An oscillatory subordination formula}

Let $L$ be a non-negative self-adjoint operator on a Lebesgue space $\mathcal{L}^2$ 
and let $0 < \sigma < 1$. In this section, we study when and how a solution to the equation
\begin{equation}
  \label{sigmawaveequation}
      \partial_t^2 u + \frac{1-2\sigma}{t}\partial_t u = -Lu, \quad t>0,
  \end{equation}
can be obtained from the oscillatory integral
\begin{equation*}
  \frac{i^\sigma t^{2\sigma}}{4^\sigma \Gamma (\sigma)} \int_0^\infty e^{-i\frac{t^2}{4s}} \, e^{-isL} \,\frac{ds}{s^{1+\sigma}} .
\end{equation*}

The main result of the article is:

\begin{theorem}
\label{mainresult1}
  Let $0 < \sigma < 1$. The limit
  \begin{equation*}
  u(\cdot , t) = \frac{i^\sigma t^{2\sigma}}{4^\sigma \Gamma (\sigma)} 
    \lim_{\substack{\varepsilon \to 0 \\ R \to \infty}} \int_\varepsilon^R e^{-i\frac{t^2}{4s}} \, e^{-isL}f 
    \,\frac{ds}{s^{1+\sigma}} , \quad t>0,
  \end{equation*}
  exists weakly in $\mathcal{L}^2$ whenever $f\in \mathcal{L}^2$ (if $\sigma \leq \frac{1}{2}$) or 
  $f\in \dom (L^{\frac{\sigma}{2} - \frac{1}{4}})$ (if $\sigma > \frac{1}{2}$).
  If $f\in\dom (L^{\frac{\sigma}{2} + \frac{3}{4}})$ then $u$ is a weak solution to equation \eqref{sigmawaveequation}
  in the sense that
  \begin{equation*}
    \Big( \partial_t^2 + \frac{1-2\sigma}{t}\partial_t \Big) \la u(\cdot , t) , h \ra = - \la Lu(\cdot , t) , h \ra
  \end{equation*}
  for all $h\in\mathcal{L}^2$. Moreover, $u$ converges to the initial data
  \begin{equation*}
    u(\cdot , 0) = f \quad \textup{and} \quad
    \partial_t^\sigma u(\cdot , 0) = -i^{2\sigma} \frac{\Gamma (1-\sigma)}{\sigma 4^\sigma \Gamma (\sigma)}
    L^\sigma f
  \end{equation*}
  weakly in $\mathcal{L}^2$. Here $\partial_t^\sigma = \frac{1}{2\sigma} t^{1-2\sigma} \partial_t$.
\end{theorem}

A solution for any combination of Dirichlet and Neumann initial data can be obtained in terms of real and imaginary parts:

\begin{cor}
\label{corollary}
  Write
  \begin{equation*}
    U_t^\sigma = \frac{i^\sigma t^{2\sigma}}{4^\sigma \Gamma (\sigma)} 
    \int_0^\infty e^{-i\frac{t^2}{4s}} \, e^{-isL} \,\frac{ds}{s^{1+\sigma}}
  \end{equation*}
  for the solution operator. Then
  \begin{itemize}
    \item the function
    \begin{equation*}
      u(\cdot , t) = \frac{1}{\sin (\sigma\pi)} \, \real (i^{1-2\sigma} U_t^\sigma (f))
    \end{equation*} 
    solves equation \eqref{sigmawaveequation} with real initial data
    \begin{equation*}
      u(\cdot , 0) = f, \quad \partial_t^\sigma u(\cdot , 0) = 0 ,
    \end{equation*}
    whenever $f\in\dom (L^{\frac{\sigma}{2} + \frac{3}{4}})$, and
    
    \item the function 
    \begin{equation*}
      u(\cdot , t) = - \frac{\sigma 4^\sigma \Gamma (\sigma)}{\sin (\sigma \pi) \Gamma (1-\sigma)} \, \imag (U_t^\sigma (L^{-\sigma}g))
    \end{equation*} 
    solves equation \eqref{sigmawaveequation} with real initial data
    \begin{equation*}
        u(\cdot , 0) = 0, \quad \partial_t^\sigma u(\cdot , 0) = g ,
    \end{equation*}
    whenever $g\in\dom (L^{-\frac{\sigma}{2} + \frac{3}{4}})$.
  \end{itemize}
\end{cor}
\begin{proof}[Proof of Corollary \ref{corollary}]
  To prove the first claim, it suffices to calculate
  \begin{equation*}
    \lim_{t\to 0} \real (i^{1-2\sigma} U_t^\sigma (f)) = \real (i^{1-2\sigma} \lim_{t\to 0}U_t^\sigma (f))
    = \real (i^{i-2\sigma} f) = \cos ((1-2\sigma)\frac{\pi}{2}) f = \sin (\sigma\pi) f
  \end{equation*}
  and to note that
  \begin{equation*}
    \lim_{t\to 0} \partial_t^\sigma \real (i^{1-2\sigma} U_t^\sigma (f))
    = \real (i^{1-2\sigma} \lim_{t\to 0} \partial_t^\sigma U_t^\sigma (f))
    = \real \Big( -i \frac{\Gamma (1-\sigma)}{\sigma 4^\sigma \Gamma (\sigma)} L^\sigma f \Big) = 0 .
  \end{equation*}
  
  For the second claim we begin by noting that
  \begin{equation*}
    \lim_{t\to 0} \imag (U_t^\sigma (L^{-\sigma}g)) = \imag (\lim_{t\to 0} U_t^\sigma (L^{-\sigma}g)) 
    = \imag (L^{-\sigma}g) = 0,
  \end{equation*}
  and calculate
  \begin{equation*}
    \lim_{t\to 0} \partial_t^\sigma \imag (U_t^\sigma (L^{-\sigma}g))
    = \imag (\lim_{t\to 0} \partial_t^\sigma U_t^\sigma (L^{-\sigma}g))
    = \imag \Big( -i^{2\sigma} \frac{\Gamma (1-\sigma)}{\sigma 4^\sigma \Gamma (\sigma)} g \Big)
    = - \sin (\sigma\pi) \frac{\Gamma (1-\sigma)}{\sigma 4^\sigma \Gamma (\sigma)} g .
  \end{equation*}
\end{proof}

The rest of the section is devoted to the proof of Theorem \ref{mainresult1}, and we start with several auxiliary results. 
Notice that the following simple lemma is false for $\sigma = 1$.

\begin{lemma}
\label{complexgamma}
  Let $0 < \sigma < 1$. Then
  \begin{equation*}
    \Gamma (\sigma) 
    = i^\sigma \lim_{R\to\infty} \int_0^R e^{-is} \,\frac{ds}{s^{1-\sigma}} .
  \end{equation*}
\end{lemma}

\begin{proof}
  For $0 < \varepsilon < R < \infty$ we may use Cauchy's integral theorem to write
  \begin{equation*}
    \int_\varepsilon^R e^{-s} \,\frac{ds}{s^{1-\sigma}} = \Big( \int_{\arc_\varepsilon} + \int_{\ray_{\varepsilon , R}} 
    + \int_{\arc_R} \Big) e^{-z} \,\frac{dz}{z^{1-\sigma}} ,
  \end{equation*}
  where $\arc_\varepsilon$ and $\arc_R$ are
  \begin{equation*}
    \theta \mapsto \varepsilon e^{i\theta} \quad \textup{and} \quad \theta \mapsto Re^{i(\frac{\pi}{2} - \theta )} , 
    \quad 0\leq \theta \leq \frac{\pi}{2} , 
  \end{equation*}
  and  $\ray_{\varepsilon , R}$ is
  \begin{equation*}
    s\mapsto is , \quad \varepsilon \leq s \leq R.
  \end{equation*}
  
  Now
  \begin{equation*}
    \Big| \int_{\arc_\varepsilon} e^{-z} \,\frac{dz}{z^{1-\sigma}} \Big|
    = \varepsilon^\sigma \Big| \int_0^{\frac{\pi}{2}} e^{-\varepsilon e^{i\theta}} e^{i\sigma\theta} \,d\theta \Big|
    \leq \varepsilon^\sigma \int_0^{\frac{\pi}{2}} e^{-\varepsilon \cos \theta} \,d\theta \leq \frac{\pi}{2} \varepsilon^\sigma
    \longrightarrow 0 ,
  \end{equation*}
  as $\varepsilon \to 0$.  
  
  Similarly,
  \begin{align*}
    \Big| \int_{\arc_R} e^{-z} \,\frac{dz}{z^{1-\sigma}} \Big|
    &= R^\sigma \Big| \int_0^{\frac{\pi}{2}} e^{-R e^{i(\frac{\pi}{2} - \theta)}} e^{i\sigma(\frac{\pi}{2} - \theta)} 
    \,d\theta \Big| 
    \leq R^\sigma \int_0^{\frac{\pi}{2}} e^{-R\sin \theta} \,d\theta \\
    &\leq R^\sigma \int_0^{\infty} e^{-\frac{2}{\pi} R\theta} \,d\theta 
    \leq R^{\sigma - 1} \int_0^{\infty} e^{-\frac{2}{\pi}\theta'} \,d\theta'
    \longrightarrow 0 ,
  \end{align*}
  as $R \to \infty$, where we used the inequality $\sin \theta \geq \frac{2}{\pi}\theta$, for $0 < \theta < \frac{\pi}{2}$.

  Thus we have
  \begin{equation*}
    \Gamma (\sigma) = \lim_{\substack{\varepsilon\to 0 \\ R\to\infty}} \int_\varepsilon^R e^{-s} \,\frac{ds}{s^{1-\sigma}}
    = \lim_{\substack{\varepsilon\to 0 \\ R\to\infty}} \int_{\ray_{\varepsilon , R}} e^{-z} \,\frac{dz}{z^{1-\sigma}}
    = i^\sigma \lim_{R\to\infty}\int_0^R e^{-is} \,\frac{ds}{s^{1-\sigma}} .
  \end{equation*}
\end{proof}

\begin{prop}
\label{lambdaeq}
  Let $0 < \sigma < 1$.
  For all $\lambda \geq 0$ the limit
  \begin{equation*}
    I_\sigma (\lambda ,t) 
    = \frac{i^\sigma}{\Gamma (\sigma)} \lim_{\substack{\varepsilon \to 0 \\ R \to \infty}}
    \int_\varepsilon^R e^{-is} e^{-i\frac{t^2}{4s}\lambda} \,\frac{ds}{s^{1-\sigma}} , \quad t>0,
  \end{equation*}
  exists and $I_\sigma (\lambda , \cdot )$ satisfies the equation
  \begin{equation}
  \label{WElambda}
    \Big( \partial_t^2 + \frac{1-2\sigma}{t}\partial_t \Big) I_\sigma (\lambda ,t) = -\lambda I_\sigma (\lambda , t) ,
  \end{equation}
  with the initial value $\lim_{t\to 0} I_\sigma (\lambda , t) = 1$. Moreover,
  \begin{equation*}
    \partial_t^\sigma I_\sigma (\lambda , t) = -i^{2\sigma} \frac{\Gamma (1-\sigma)}{\sigma 4^\sigma \Gamma (\sigma)}
    \lambda^\sigma I_{1-\sigma}(\lambda , t) ,
  \end{equation*}
  where $\partial_t^\sigma = \frac{1}{2\sigma} t^{1-2\sigma} \partial_t$.
\end{prop}

\begin{proof}
  The proof is split into four parts.
  
  \textbf{Part I}:
  We consider the truncated integral
  \begin{equation}
  \label{integral}
    \int_\varepsilon^R e^{-is} e^{-i\frac{t^2}{4s}\lambda} \,\frac{ds}{s^{1-\sigma}} ,
  \end{equation}
  its convergence as $\varepsilon \to 0$ and $R\to\infty$, and to what extent it gives a solution to equation \eqref{WElambda}.
  Note first that the integral is absolutely convergent as $\varepsilon \to 0$.
  On the other hand, using integration by parts, we see that
  \begin{equation}
  \label{intbyparts}
  \begin{split}
    \int_\varepsilon^R e^{-is} e^{-i\frac{t^2}{4s}\lambda} \,\frac{ds}{s^{1-\sigma}}
    &= i \int_\varepsilon^R \partial_s (e^{-is}) e^{-i\frac{t^2}{4s}\lambda} \,\frac{ds}{s^{1-\sigma}} \\
    &= i \Big[ e^{-is} e^{-i\frac{t^2}{4s}\lambda} s^{\sigma -1} \Big]_{s=\varepsilon}^R
    -i\int_\varepsilon^R e^{-is} \partial_s (e^{-i\frac{t^2}{4s}\lambda} s^{\sigma - 1}) \,ds .
  \end{split}
  \end{equation}
  Here
  \begin{equation}
  \label{sderivative}
    \partial_s (e^{-i\frac{t^2}{4s}\lambda} s^{\sigma - 1}) = (i\frac{t^2}{4}\lambda s^{\sigma - 3} - (1-\sigma)s^{\sigma - 2})
    e^{-i\frac{t^2}{4s}\lambda}
  \end{equation}
  so that the last integral converges absolutely as $R\to\infty$. Since the integrated term
  converges as $R\to\infty$, we see that also the integral \eqref{integral} converges.
  
  We then study the convergence of the integral corresponding to the time derivatives, and calculate:
  \begin{equation}
  \label{tderivative}
    \Big( \partial_t^2 + \frac{1-2\sigma}{t}\partial_t \Big) e^{-i\frac{t^2}{4s}\lambda}
    = \Big( -\frac{t^2}{4s^2}\lambda^2 - i\frac{1-\sigma}{s} \lambda \Big) e^{-i\frac{t^2}{4s}\lambda} ,
  \end{equation}
  so that
  \begin{equation*}
    \int_\varepsilon^R e^{-is} \Big( \partial_t^2 + \frac{1-2\sigma}{t}\partial_t \Big) e^{-i\frac{t^2}{4s}\lambda}
    \,\frac{ds}{s^{1-\sigma}}
    = -\frac{t^2}{4}\lambda^2 \int_\varepsilon^R e^{-is} e^{-i\frac{t^2}{4s}\lambda} \,\frac{ds}{s^{3-\sigma}}
    -i(1-\sigma) \lambda \int_\varepsilon^R e^{-is} e^{-i\frac{t^2}{4s}\lambda} \,\frac{ds}{s^{2-\sigma}} ,
  \end{equation*}
  which converges absolutely as $R\to\infty$, but appears problematic for $\varepsilon\to 0$ (see Remark \ref{divrem}).
  In order to see to what extent the truncated integrals \eqref{integral} solve equation \eqref{WElambda}, we note from
  \eqref{sderivative} and \eqref{tderivative} that
  \begin{equation}
  \label{stderivative}
    \Big( \partial_t^2 + \frac{1-2\sigma}{t}\partial_t \Big) e^{-i\frac{t^2}{4s}\lambda} s^{\sigma - 1}
    = i\lambda \partial_s (e^{-i\frac{t^2}{4s}\lambda}s^{\sigma - 1}) .
  \end{equation}
  In the light of equations \eqref{intbyparts} and \eqref{stderivative}, we infer that
  \begin{align*}
    \Big( \partial_t^2 + \frac{1-2\sigma}{t}\partial_t \Big) 
    \int_\varepsilon^R e^{-is} e^{-i\frac{t^2}{4s}\lambda} \,\frac{ds}{s^{1-\sigma}}
    &= i\lambda \int_\varepsilon^R e^{-is} \partial_s (e^{-i\frac{t^2}{4s}\lambda} s^{\sigma - 1}) \,ds \\
    &= -\lambda \int_\varepsilon^R e^{-is} e^{-i\frac{t^2}{4s}\lambda} \,\frac{ds}{s^{1-\sigma}}
    + i\lambda \Big[ e^{-is} e^{-i\frac{t^2}{4s}\lambda} s^{\sigma -1} \Big]_{s=\varepsilon}^R .
  \end{align*}
  At the upper limit $R\to\infty$ the integrated term vanishes, but at the lower limit $\varepsilon\to 0$
  it diverges and a careful argument will be needed. 
  For now we record that
  \begin{equation*}
    \Big( \partial_t^2 + \frac{1-2\sigma}{t}\partial_t + \lambda \Big) 
    \int_1^\infty e^{-is} e^{-i\frac{t^2}{4s}\lambda} \,\frac{ds}{s^{1-\sigma}}
    = -i\lambda e^{-i} e^{-i\frac{t^2}{4}\lambda} .
  \end{equation*}
  Note here that the integral \eqref{integral} together with its time derivatives converges locally uniformly in $t$,
  which justifies differentiating under the integral sign.
  
  \textbf{Part II}:
  In order to show that
  \begin{equation}
  \label{zeroone}
    \Big( \partial_t^2 + \frac{1-2\sigma}{t}\partial_t + \lambda \Big) 
    \int_0^1 e^{-is} e^{-i\frac{t^2}{4s}\lambda} \,\frac{ds}{s^{1-\sigma}}
    = i\lambda e^{-i} e^{-i\frac{t^2}{4}\lambda} ,
  \end{equation}
  from which \eqref{WElambda} follows,
  we view the question in terms of complex path integrals:
  \begin{equation*}
    \int_0^1 e^{-is} e^{-i\frac{t^2}{4s}\lambda} \,\frac{ds}{s^{1-\sigma}} = i^{-\sigma}
    \int_0^1 e^{-(is)}e^{\frac{t^2}{4(is)}\lambda} \,\frac{d(is)}{(is)^{1-\sigma}}
    =i^{-\sigma} \lim_\gamma \int_\gamma e^{-z}e^{\frac{t^2}{4z}\lambda} \,\frac{dz}{z^{1-\sigma}} ,
  \end{equation*}
  where the paths $\gamma$ approximate the segment $[0,i]$ in a suitable way (which we elaborate on below).
  
  Let $0 < \delta < \frac{\pi}{2}$ be fixed and consider, for $0 < \varepsilon < 1$, the paths $\arc_\varepsilon$, $\arc_1$ and
  $\ray_\varepsilon$ given by
  \begin{equation*}
    \theta \mapsto \varepsilon e^{i\theta} \quad \textup{and} \quad \theta \mapsto e^{i\theta} , 
    \quad \frac{\pi}{2} \leq \theta \leq \frac{\pi}{2} + \delta , 
  \end{equation*}
  and
  \begin{equation*}
    s\mapsto se^{i(\frac{\pi}{2} + \delta)} , \quad \varepsilon \leq s \leq 1.
  \end{equation*}  
  By Cauchy's integral theorem we may now write
  \begin{equation*}
    \int_\varepsilon^1 e^{-is}e^{-i\frac{t^2}{4s}\lambda} \,\frac{ds}{s^{1-\sigma}}
    = i^{-\sigma} \Big( \int_{\arc_\varepsilon} + \int_{\ray_\varepsilon} - \int_{\arc_1} \Big)
    e^{-z} e^{\frac{t^2}{4z}\lambda} \,\frac{dz}{z^{1-\sigma}} .
  \end{equation*}
  
  To see that the first integral on the right-hand side vanishes as $\varepsilon \to 0$, it suffices to observe that
  $\real (e^{\pm i\theta}) = \cos \theta \leq 0$ when $\frac{\pi}{2} \leq \theta \leq \frac{\pi}{2} + \delta$ so that 
  \begin{equation*}
    |e^{-\varepsilon e^{i\theta}} e^{\frac{t^2}{4\varepsilon}\lambda e^{-i\theta}}|
    = e^{-\varepsilon \cos \theta} e^{\frac{t^2}{4\varepsilon}\lambda \cos \theta} \leq e^{-\varepsilon \cos\theta}
    \lesssim 1 .
  \end{equation*}
  
  The second integral on the right-hand side is
  \begin{equation*}
    \int_{\ray_\varepsilon} e^{-z} e^{\frac{t^2}{4z}\lambda} \,\frac{dz}{z^{1-\sigma}}
    = e^{i\sigma (\frac{\pi}{2} + \delta)} 
    \int_\varepsilon^1 e^{-se^{i(\frac{\pi}{2} + \delta)}} e^{\frac{t^2}{4s} \lambda e^{-i(\frac{\pi}{2} + \delta)}} 
    \,\frac{ds}{s^{1-\sigma}} ,
  \end{equation*}
  and it converges to a limit $J_\delta (\lambda , t)$ as $\varepsilon \to 0$, again because
  $\real (e^{-i(\frac{\pi}{2} + \delta)}) = \cos (\frac{\pi}{2} + \delta) < 0$.
  
  Lastly,  
  \begin{equation*}
    \int_{\arc_1} e^{-z} e^{\frac{t^2}{4z}\lambda} \,\frac{dz}{z^{1-\sigma}}
    = i \int_{\frac{\pi}{2}}^{\frac{\pi}{2} + \delta} e^{-e^{i\theta}} e^{\frac{t^2}{4} \lambda e^{-i\theta}} 
    e^{i\sigma\theta} \,d\theta
    =: E_\delta (\lambda , t),
  \end{equation*}
  so that altogether we have
  \begin{equation*}
    \int_0^1 e^{-is} e^{-i\frac{t^2}{4s}\lambda} \,\frac{ds}{s^{1-\sigma}} 
    = i^{-\sigma} (J_\delta(\lambda , t) - E_\delta (\lambda , t)) .
  \end{equation*}
  
  \textbf{Part III}:  
  For \eqref{zeroone} it suffices to show that
  \begin{equation*}
    \Big( \partial_t^2 + \frac{1-2\sigma}{t}\partial_t + \lambda \Big) E_\delta (\lambda , t) 
    \longrightarrow 0 , \quad \textup{as} \quad \delta \to 0
  \end{equation*}
  and that
  \begin{equation}
  \label{Jlimit}
    \Big( \partial_t^2 + \frac{1-2\sigma}{t}\partial_t + \lambda \Big) J_\delta (\lambda , t) 
    \longrightarrow i^{\sigma +1}\lambda e^{-i}e^{-i\frac{t^2}{4}\lambda} , \quad \textup{as} \quad \delta \to 0.
  \end{equation}
  
  It is easy to see that
  \begin{equation*}
    \Big| \Big( \partial_t^2 + \frac{1-2\sigma}{t}\partial_t + \lambda \Big) E_\delta (\lambda , t) \Big|
    \lesssim (t^2\lambda^2 + \lambda) \int_{\frac{\pi}{2}}^{\frac{\pi}{2}+\delta}
    e^{-\cos\theta} e^{\frac{t^2}{4}\lambda\cos\theta} \,d\theta \longrightarrow 0 , \quad \textup{as} \quad \delta\to 0.
  \end{equation*}
  
  Write $\theta_\delta = \frac{\pi}{2} + \delta$ so that
  \begin{equation*}
    J_\delta (\lambda , t) = e^{i\sigma \theta_\delta} 
    \int_0^1 e^{-se^{i\theta_\delta}} e^{\frac{t^2}{4s} \lambda e^{-i\theta_\delta}} \,\frac{ds}{s^{1-\sigma}} .
  \end{equation*}  
  Replacing $\lambda$ by $i\lambda e^{-i\theta_\delta}$ in \eqref{stderivative} and \eqref{intbyparts}, we find that
  \begin{align*}
    \Big( \partial_t^2 + \frac{1-2\sigma}{t}\partial_t \Big) J_\delta (\lambda , t)
    &= -e^{i(\sigma -1) \theta_\delta} \lambda \int_0^1 e^{-se^{i\theta_\delta}} 
    \partial_s (e^{\frac{t^2}{4s}\lambda e^{-i\theta_\delta}} s^{\sigma - 1}) \,ds \\
    &= -\lambda e^{i(\sigma -1) \theta_\delta} 
    \Big[ e^{-se^{i\theta_\delta}} e^{\frac{t^2}{4s}\lambda e^{-i\theta_\delta}} s^{\sigma -1} \Big]_{s=0}^1
    -\lambda J_\delta (\lambda , t) ,
  \end{align*}
  where the integrated term poses no problem at $s=0$ because
  $\real (e^{-i\theta_\delta}) = \cos \theta_\delta < 0$. 
  As $\delta\to 0$, the integrated term therefore tends to $\lambda i^{\sigma + 1} e^{-i}e^{-i\frac{t^2}{4}\lambda}$,
  and \eqref{Jlimit} and \eqref{zeroone} follow.

\textbf{Part IV}:
We address the convergence to the initial value. By Lemma \ref{complexgamma},
\begin{equation*}
  I_\sigma (\lambda , t ) - 1 = \frac{i^\sigma}{\Gamma (\sigma)} \lim_{\varepsilon , R} \int_\varepsilon^R
  e^{-is}(e^{-i\frac{t^2}{4s}\lambda} - 1) \,\frac{ds}{s^{1-\sigma}} .
\end{equation*}
The integral above is absolutely convergent as $\varepsilon \to 0$ and therefore by dominated convergence,
\begin{equation*}
  \int_0^1 e^{-is}(e^{-i\frac{t^2}{4s}\lambda} - 1) \,\frac{ds}{s^{1-\sigma}} \longrightarrow 0 \quad \textup{as} \quad t\to 0.
\end{equation*}
Using integration by parts, we see that
\begin{align*}
  \int_1^R e^{-is}(e^{-i\frac{t^2}{4s}\lambda} - 1) \,\frac{ds}{s^{1-\sigma}} &=
  i\int_1^R \partial_s (e^{-is}) (e^{-i\frac{t^2}{4s}\lambda} - 1) \,\frac{ds}{s^{1-\sigma}} \\
  &=i \Big[ e^{-is}(e^{-i\frac{t^2}{4s}\lambda} - 1) s^{\sigma - 1} \Big]_{s=1}^R
  -i \int_1^R e^{-is} \partial_s ( (e^{-i\frac{t^2}{4s}\lambda} - 1)s^{\sigma - 1}) \,ds ,
\end{align*}
where
\begin{equation*}
  \partial_s ( (e^{-i\frac{t^2}{4s}\lambda} - 1)s^{\sigma - 1})
  = i\frac{t^2}{4}\lambda e^{-i\frac{t^2}{4s}\lambda} s^{\sigma - 3} - (1-\sigma) (e^{-i\frac{t^2}{4s}\lambda} - 1) s^{\sigma - 2} ,
\end{equation*}
so at the limit $R\to\infty$ we have
\begin{align*}
  \int_1^\infty e^{-is}(e^{-i\frac{t^2}{4s}\lambda} - 1) \,\frac{ds}{s^{1-\sigma}}
  = &-ie^{-i}(e^{-i\frac{t^2}{4}\lambda} - 1) 
  + \frac{t^2}{4}\lambda \int_1^\infty e^{-is} e^{-i\frac{t^2}{4s}\lambda} \,\frac{ds}{s^{3-\sigma}} \\
  &+ i(1-\sigma) \int_1^\infty e^{-is}(e^{-i\frac{t^2}{4s}\lambda} - 1) \,\frac{ds}{s^{2-\sigma}} .
\end{align*}
The two integrals on the right-hand side converge absolutely and thus by dominated convergence we have
\begin{equation*}
  \int_1^\infty e^{-is}(e^{-i\frac{t^2}{4s}\lambda} - 1) \,\frac{ds}{s^{1-\sigma}} \longrightarrow 0
  \quad \textup{as} \quad t\to 0.
\end{equation*}
This shows that $I_\sigma (\lambda ,t) \to 1$ as $t\to 0$.

Performing a change of variables $\frac{t^2}{4s}\lambda = r$ in the truncated integral, we obtain
\begin{align*}
  \partial_t \int_\varepsilon^R e^{-is}e^{-i\frac{t^2}{4s}\lambda} \,\frac{ds}{s^{1-\sigma}}
  &=-i \frac{t}{2} \lambda \int_\varepsilon^R e^{-is} e^{-i\frac{t^2}{4s}\lambda} \,\frac{ds}{s^{2-\sigma}} \\
  &=-i \frac{4^{1-\sigma}}{2} \frac{\lambda^\sigma}{t^{1-2\sigma}} \int_{t^2\lambda / 4R}^{t^2\lambda / 4\varepsilon}
  e^{-i\frac{t^2}{4r}\lambda} \, e^{-ir} \,\frac{dr}{r^{1 - (1-\sigma)}} .
\end{align*}
Since the convergence of the last integral here is locally uniform in $t$, it is easy to verify that
\begin{align*}
  \partial_t^\sigma I_\sigma (\lambda , t) &= \frac{t^{1-2\sigma}}{2\sigma} \partial_t I_\sigma (\lambda , t) \\
  &= -\frac{i^{\sigma + 1}}{\Gamma (\sigma)} \frac{\lambda^\sigma}{\sigma 4^\sigma} 
  \lim_{\varepsilon' , R'} \int_{\varepsilon'}^{R'} e^{-i\frac{t^2}{4r}\lambda} \, e^{-ir} \,\frac{dr}{r^{1 - (1-\sigma)}} \\
  &= -i^{2\sigma} \frac{\Gamma (1-\sigma)}{\sigma 4^\sigma \Gamma (\sigma)} \lambda^\sigma I_{1-\sigma}(\lambda , t) .
\end{align*}

\end{proof}

\begin{remark}
\label{divrem}
  As $\varepsilon\to 0$, the integral
  \begin{equation*}
    \int_\varepsilon^R e^{-is} e^{-i\frac{t^2}{4s}\lambda} \,\frac{ds}{s^{k-\sigma}}
  \end{equation*}
  converges for $k=2$ but diverges for $k=3$. In both cases, this can be seen by means of the change of variable
  $r = \frac{t^2}{4s}\lambda$. For $k=2$, we also get bounds for the integral by means of Lemma \ref{firstestimate}.
\end{remark}

\begin{lemma}
\label{firstestimate}
  Let $0 < \sigma < 1$. For every $t>0$ we have
  \begin{equation}
  \label{noder}
    \Big| \int_\varepsilon^R e^{-is} e^{-i\frac{t^2}{4s}\lambda} \,\frac{ds}{s^{1-\sigma}} \Big| \lesssim 
    \begin{cases}
      1 , &\sigma\leq \frac{1}{2} , \\
      1 + \lambda^{\frac{\sigma}{2} - \frac{1}{4}} , &\sigma > \frac{1}{2} ,
    \end{cases}
  \end{equation}
  uniformly in $\varepsilon$ and $R$.
\end{lemma}  
  
  \begin{proof}
  We denote $A = \frac{t\sqrt{\lambda}}{2}$ and show that
  \begin{equation*}
    \Big| \int_\varepsilon^R e^{-i(s+\frac{A^2}{s})} \,\frac{ds}{s^{1-\sigma}} \Big|
    \lesssim
    \begin{cases}
      1 , &A < 1 , \\
      1 + A^{\sigma - \frac{1}{2}} , &A \geq 1 .
    \end{cases}
  \end{equation*}
  Observe that the phase $\varphi (s) = s + \frac{A^2}{s}$ has its critical 
  point where $0 = \varphi'(s) = 1 - \frac{A^2}{s^2}$ at $s = A$.
  
  Consider first the case $A < 1$.
  We write
  \begin{equation*}
    \Big( \int_\varepsilon^2 + \int_2^R \Big) e^{-i(s+\frac{A^2}{s})} \,\frac{ds}{s^{1-\sigma}} =: I_1 + I_2 ,
  \end{equation*}
  and note that $|I_1| \lesssim 1$. Using integration by parts, we see that
  \begin{equation*}
    I_2 = i\int_2^R \partial_s (e^{-i\varphi(s)}) \,\frac{ds}{\varphi'(s) s^{1-\sigma}} 
    = i \Big[ e^{-i\varphi(s)} \frac{s^{\sigma + 1}}{s^2 - A^2} \Big]_{s=2}^R
    -i \int_2^R e^{-i\varphi(s)} \partial_s \Big( \frac{s^{\sigma + 1}}{s^2 - A^2} \Big) \,ds .
  \end{equation*}
  Since $A^2 \leq 1$, we have
  \begin{equation*}
    \Big| \Big[ e^{-i\varphi(s)} \frac{s^{\sigma + 1}}{s^2 - A^2} \Big]_{s=2}^R \Big|
    \leq \frac{2^{\sigma + 1}}{3} + \frac{R^{\sigma + 1}}{R^2 - 1}
    \lesssim 1 
  \end{equation*}
  and upon calculating the derivative
  \begin{equation*}
    \partial_s \Big( \frac{s^{\sigma + 1}}{s^2 - A^2} \Big) = \frac{(\sigma + 1) s^\sigma}{s^2 - A^2}
    - \frac{2s^{\sigma + 2}}{(s^2 - A^2)^2}
  \end{equation*}
  we get
  \begin{equation*}
    \Big| \int_2^R e^{-i\varphi(s)} \partial_s \Big( \frac{s^{\sigma + 1}}{s^2 - A^2} \Big) \,ds \Big|
    \lesssim \int_2^\infty \frac{s^\sigma}{s^2 - A^2} ds + \int_2^\infty \frac{s^{\sigma + 2}}{(s^2 - A^2)^2} \,ds \lesssim 1 ,
  \end{equation*}
  so that $| I_2 | \lesssim 1$. 
  
  Consider then the case $A\geq 1$.
  We follow the principle of stationary phase and decompose the integral into three pieces
  \begin{equation*}
    \Big( \int_\varepsilon^{A - \frac{1}{2}\sqrt{A}} + \int_{A - \frac{1}{2}\sqrt{A}}^{A + \frac{1}{2}\sqrt{A}}
     + \int_{A + \frac{1}{2}\sqrt{A}}^R \Big)
    e^{-i(s+\frac{A^2}{s})} \,\frac{ds}{s^{1-\sigma}} =: I_1 + I_2 + I_3 ;
  \end{equation*}
  observe that $A\pm \frac{1}{2}\sqrt{A} \sim A$.
  
  Note first that
  \begin{equation*}
    |I_2| \leq \int_{A - \frac{1}{2}\sqrt{A}}^{A + \frac{1}{2}\sqrt{A}} \,\frac{ds}{s^{1-\sigma}}
    \leq \frac{\sqrt{A}}{(A - \frac{1}{2}\sqrt{A})^{1 - \sigma}} \lesssim A^{\sigma - \frac{1}{2}}. 
  \end{equation*}
  For $I_1 + I_3$ we use integration by parts:
  \begin{align*}
    I_1 + I_3 &= \Big( \int_\varepsilon^{A-\frac{1}{2}\sqrt{A}} + \int_{A+\frac{1}{2}\sqrt{A}}^R \Big)
    e^{-i\varphi(s)} \,\frac{ds}{s^{1-\sigma}} \\
    &=i \Big( \int_\varepsilon^{A-\frac{1}{2}\sqrt{A}} + \int_{A+\frac{1}{2}\sqrt{A}}^R \Big)
    \partial_s (e^{-i\varphi(s)}) \,\frac{ds}{\varphi'(s)s^{1-\sigma}} \\
    &=i\Big[ e^{-i\varphi(s)} \frac{s^{\sigma + 1}}{s^2 -A^2} \Big]_{s=\varepsilon}^{A-\frac{1}{2}\sqrt{A}}
    +i\Big[ e^{-i\varphi(s)} \frac{s^{\sigma + 1}}{s^2 -A^2} \Big]_{s=A+\frac{1}{2}\sqrt{A}}^R \\
    &-i \Big( \int_\varepsilon^{A-\frac{1}{2}\sqrt{A}} + \int_{A+\frac{1}{2}\sqrt{A}}^R \Big) e^{-i\varphi(s)} 
    \partial_s \Big( \frac{s^{1+\sigma}}{s^2 - A^2} \Big) \,ds .
  \end{align*}
  Here
  \begin{equation*}
    \Big| \Big[ e^{-i\varphi(s)} \frac{s^{\sigma + 1}}{s^2 -A^2} \Big]_{s=\varepsilon}^{A-\frac{1}{2}\sqrt{A}} \Big|
    \leq \frac{(A-\frac{1}{2}\sqrt{A})^{\sigma + 1}}{A^2 - (A-\frac{1}{2}\sqrt{A})^2} 
    + \frac{\varepsilon^{\sigma + 1}}{A^2 - \varepsilon^2} \lesssim 1 + A^{\sigma - \frac{1}{2}} .
  \end{equation*}
  and similarly
  \begin{equation*}
    \Big| \Big[ e^{-i\varphi(s)} \frac{s^{\sigma + 1}}{s^2 -A^2} \Big]_{s=A+\frac{1}{2}\sqrt{A}}^R \Big|
    \leq \frac{(A+\frac{1}{2}\sqrt{A})^{\sigma + 1}}{(A+\frac{1}{2}\sqrt{A})^2 - A^2} + \frac{R^{\sigma + 1}}{R^2 - A^2}
    \lesssim 1 + A^{\sigma - \frac{1}{2}} . 
  \end{equation*}
  Moreover, because
  \begin{equation*}
    \partial_s \Big( \frac{s^{1+\sigma}}{s^2 - A^2} \Big) 
    = \frac{(s^2 - A^2)(1+\sigma)s^\sigma - 2s^{\sigma + 2}}{(s^2 - A^2)^2}
    = \frac{s^{\sigma + 2}(\sigma - 1) - A^2(1+\sigma)s^\sigma}{(s^2 - A^2)^2} < 0,
  \end{equation*}
  we have
  \begin{align*}
    \Big| \Big( \int_\varepsilon^{A-\frac{1}{2}\sqrt{A}} + \int_{A+\frac{1}{2}\sqrt{A}}^R \Big) e^{-i\varphi(s)} 
    \partial_s \Big( \frac{s^{1+\sigma}}{s^2 - A^2} \Big) \,ds \Big|
    &\leq - \Big( \int_\varepsilon^{A-\frac{1}{2}\sqrt{A}} + \int_{A+\frac{1}{2}\sqrt{A}}^R \Big) 
    \partial_s \Big( \frac{s^{1+\sigma}}{s^2 - A^2} \Big) \,ds \\
    &= -\Big[ \frac{s^{1+\sigma}}{s^2 - A^2} \Big]_{s=\varepsilon}^{A-\frac{1}{2}\sqrt{A}}
    -\Big[ \frac{s^{1+\sigma}}{s^2 - A^2} \Big]_{s=A+\frac{1}{2}\sqrt{A}}^R \\
    &\lesssim 1 + A^{\sigma - \frac{1}{2}} , 
  \end{align*}
  as before.
  Therefore
  \begin{equation*}
    |I_1 + I_3| \lesssim 1 + A^{\sigma - \frac{1}{2}} .
  \end{equation*}
\end{proof}


\begin{lemma}
\label{secondestimate}
  Let $0 < \sigma < 1$. For every $t>0$ we have
  \begin{equation}
  \label{1stder}
    \Big| \Big( \partial_t^2 + \frac{1-2\sigma}{t}\partial_t \Big) I_\sigma (\lambda , t) \Big| 
    \lesssim \max \{ 1, \lambda^{\frac{\sigma}{2} + \frac{3}{4}} \} .
  \end{equation}
\end{lemma}

\begin{proof}
In order to obtain estimates for the derivatives
\begin{equation*}
  \Big( \partial_t^2 + \frac{1-2\sigma}{t}\partial_t \Big) 
  \int_0^\infty e^{-i(s+\frac{t^2}{4s}\lambda )} \,\frac{ds}{s^{1-\sigma}} ,
\end{equation*}
we now split the interval of integration in a way which does not
depend on $t$. We shall  prove the estimate for $t$ in a small
neighbourhood of  a fixed point $t_0>0$,
by moving the differentiations into the integrals. 
This will be uniform in  $t_0>0$.
With  $A = {t\sqrt{\lambda}}/{2}$ as before,
 we write  $A_0$ for the value corresponding to $t_0$.
 Let $d = \min(A_0, \sqrt{A_0})/2$. The integral will be
split at the points  $a = A_0- d$ and  $b = A_0+ d$. We consider only
values of 
  $t$ so close to  $t_0$ that $|A-A_0|<d/2$. Thus
$|s-A|> d/2 \sim\min(A, \sqrt{A}) $ as soon as
$s\notin (a,b)$. 

 For such  $t$  we write
  \begin{equation*}
    \Big( \partial_t^2 + \frac{1-2\sigma}{t}\partial_t \Big) \Big( \int_0^{a} + \int_{a}^{b} 
    + \int_{b}^\infty \Big)
    e^{-i(s+\frac{t^2}{4s}\lambda)} \,\frac{ds}{s^{1-\sigma}} = I_1 + I_2 + I_3 ,
  \end{equation*}
say.
In $I_2$ and  $I_3$, it is clear that the derivatives can be taken inside the integrals:
  \begin{align*}
     &\Big( \partial_t^2 + \frac{1-2\sigma}{t}\partial_t \Big) 
    \Big( \int_{a}^{b} + \int_{b}^\infty \Big) e^{-i(s+\frac{t^2}{4s}\lambda)} \,\frac{ds}{s^{1-\sigma}}
  \\  &= -i(1-\sigma) \lambda \Big( \int_{a}^{b} + \int_{b}^\infty \Big) e^{-i(s+\frac{A^2}{s})} \,\frac{ds}{s^{2-\sigma}} 
    - \frac{t^2\lambda^2}{4} \Big( \int_{a}^{b} + \int_{b}^\infty \Big) e^{-i(s+\frac{A^2}{s})} \,\frac{ds}{s^{3-\sigma}}.
  \end{align*}

To estimate  $I_2$, we observe that $s\sim A$ for $a\le s\le b$ and that 
$b-a =2d \lesssim \sqrt A \sim \lambda^{1/4}$. Thus 
  \begin{equation*}
    |I_2| \leq \lambda \int_{a}^{b} \,\frac{ds}{s^{2-\sigma}}
    + t^2 \lambda^2 \int_{a}^{b} \,\frac{ds}{s^{3-\sigma}}
    \leq \lambda \,\frac{b-a}{A^{2-\sigma}} + t^2\lambda^2\, \frac{b-a}{A^{3-\sigma}}
    \lesssim \lambda^{\frac{\sigma}{2} + \frac{1}{4}} + \lambda^{\frac{\sigma}{2} + \frac{3}{4}}.
  \end{equation*}

 For $I_3$ the estimate is clear when $A \le  1$:
  \begin{align*}
    |I_3| &\leq \lambda \int_b^\infty \,\frac{ds}{s^{2-\sigma}} 
    + t^2\lambda^2 \int_b^\infty \,\frac{ds}{s^{3-\sigma}} \\
    &\lesssim \lambda A^{\sigma - 1} + \lambda^2 A^{\sigma - 2} 
    \leq \lambda A^{\sigma - \frac{3}{2}} + \lambda^2 A^{\sigma - \frac{5}{2}} \\
   & \lesssim \lambda^{\frac{\sigma}{2} + \frac{1}{4}} + 
\lambda^{\frac{\sigma}{2} + \frac{3}{4}} 
  \lesssim \max(1,\lambda^{\frac{\sigma}{2} + \frac{3}{4}}).
  \end{align*}

 When $A > 1$ we  integrate by parts, with $k=2,3$:
  \begin{align*}
    \int_{b}^\infty e^{-i\varphi(s)} \,\frac{ds}{s^{k-\sigma}}
    &=i \int_{b}^\infty \partial_s (e^{-i\varphi(s)}) \,\frac{ds}{\varphi'(s)s^{k-\sigma}} \\
    &=i\Big[ e^{-i\varphi(s)} \frac{s^{\sigma - k + 2}}{s^2 -A^2} \Big]_{s=b}^\infty 
    -i \int_{b}^\infty e^{-i\varphi(s)} \partial_s \Big( \frac{s^{\sigma - k + 2}}{s^2 - A^2} \Big) \,ds .
  \end{align*}
  Here
  \begin{equation*}
    \Big| \Big[ e^{-i\varphi(s)} \frac{s^{\sigma - k + 2}}{s^2 -A^2} \Big]_{s=b}^\infty \Big|
    = \frac{b^{\sigma - k + 2}}{b^2 - A^2} 
    \sim \frac{A^{\sigma - k + 2}}{A \sqrt{A}}    
    \lesssim A^{\sigma - k + \frac{1}{2}} .
  \end{equation*}
  Moreover, since by simple calculus,
  \begin{align*}
    \partial_s \Big( \frac{s^{\sigma - k + 2}}{s^2 - A^2} \Big) 
    &= \frac{(s^2 - A^2)(\sigma - k + 2) s^{\sigma - k + 1} - 2s^{\sigma -k + 3}}{(s^2 - A^2)^2} \\
    &= \frac{(\sigma - k)s^{\sigma - k + 3} - A^2(\sigma - k + 2)  s^{\sigma - k + 1}}{(s^2 - A^2)^2} < 0
  \end{align*}
  for all $s > A$, we have

  \begin{align*}
    \Big| \int_{b}^\infty  e^{-i\varphi(s)} 
    \partial_s \Big( \frac{s^{\sigma - k + 2}}{s^2 - A^2} \Big) \,ds \Big|
    \leq - \int_{b}^\infty  \partial_s \Big( \frac{s^{\sigma - k + 2}}{s^2 - A^2} \Big) \,ds
 = -\Big[ \frac{s^{\sigma - k + 2}}{s^2 - A^2} \Big]_{s=b}^\infty 
    \lesssim A^{\sigma - k + \frac{1}{2}} , 
 \end{align*}
as before.
  Therefore
  \begin{align*}
    |I_3| &\leq \lambda \Big| \int_{b}^\infty e^{-i\varphi(s)} \frac{ds}{s^{2-\sigma}} \Big|
    + t^2\lambda^2 \Big| \int_{b}^\infty e^{-i\varphi(s)} \,\frac{ds}{s^{3-\sigma}} \Big| \\
    &\lesssim \lambda A^{\sigma -\frac{3}{2}} + \lambda^2 A^{\sigma -\frac{5}{2}} 
   \lesssim \max(1,\lambda^{\frac{\sigma}{2} + \frac{3}{4}}), 
  \end{align*}
as before.

 We handle $I_1$ by switching to complex path integrals as in the proof of Proposition \ref{lambdaeq}, namely
  \begin{equation*}
    \int_0^{a} e^{-i(s+\frac{t^2}{4s}\lambda)} \,\frac{ds}{s^{1-\sigma}}
  = i^{-\sigma} \int_{[0,ia]} e^{-z+\frac{A^2}z}\,\frac{dz}{z^{1-\sigma}}
    = i^{-\sigma} \Big( \int_{\ray_a} - \int_{\arc_a} \Big) e^{-z+\frac{A^2}{z}} \,\frac{dz}{z^{1-\sigma}} ,
  \end{equation*}
  where
  \begin{equation*}
    \ray_a: \quad s\mapsto se^{i\frac{2\pi}{3}} , \quad 0 < s \leq a,
  \end{equation*}
  and
  \begin{equation*}
    \arc_a: \quad \theta \mapsto a e^{i\left(\theta+\frac\pi 2\right)} , 
\quad 0 \leq \theta \leq \frac{\pi}{6} . 
  \end{equation*}

In the integrals over  $\ray_a$ and  $\arc_a$, we shall take the
derivatives of the integrand and verify convergence of the resulting
integrals. Observe first that for $0 \leq \theta \leq {\pi}/{6}$
  \begin{equation*}
    |e^{-a e^{i(\theta+\pi/2)} + \frac{A^2}{a}e^{-i(\theta+\pi/2)}}| 
    = e^{-\left(-a + \frac{A^2}a\right) \sin \theta},   
     \end{equation*}
and here
\[
-s+\frac{A^2}s =  2(A-s) + \frac{(A-s)^2}s \gtrsim d + \frac{d^2}s.
\]
For  $\theta = {\pi}/{6}$ this implies that 
 \begin{equation*}
    |e^{-se^{i{2\pi}/{3}} + \frac{A^2}{s} e^{-i{2\pi}/{3}}}|
 = e^{-\frac1{2}(-s+\frac{A^2}{s})}
 \lesssim e^{-c\sqrt A} \, e^{-c\,\frac{\min(A^2, A)}s} .
  \end{equation*}

 With $k=2,3$, we get for $A \le 1$
\begin{equation*}
   \int_{\ray_a}  \Big| e^{-z + \frac{A^2}{z}} \,\frac{1}{z^{k-\sigma}} \Big|\,{dz}
   \lesssim \int_0^{a}   e^{-c\,\frac{A^2}s}  \,\frac{ds}{s^{k-\sigma}} 
 \lesssim A^{2\sigma-2k+2};
\end{equation*}
and for  $A > 1$
\begin{equation*}
   \int_{\ray_a}  \Big| e^{-z + \frac{A^2}{z}} \,\frac{1}{z^{k-\sigma}} \Big|\,{dz}
   \leq e^{-c\sqrt A} \int_0^{a}   e^{-c\,\frac{A}s}  \,\frac{ds}{s^{k-\sigma}} 
 \lesssim A^{-M},
\end{equation*}
for any $M$.

We also get
\begin{equation*}
    \Big| \int_{\arc_a} e^{-z + \frac{A^2}{z}} \,\frac{dz}{z^{k-\sigma}} \Big| 
    \lesssim A^{\sigma - k + 1} \int_0^{\frac{\pi}{6}} e^{-c\sqrt{A}\,\theta} \,d\theta
    \lesssim A^{\sigma - k + \frac{1}{2}},
  \end{equation*}
for all $A$.

  Altogether, this implies that
  \begin{align*}
    |I_1| &\leq \lambda \Big| \Big( \int_{\ray_a} - \int_{\arc_a} \Big) e^{-z+\frac{A^2}{z}} \,\frac{dz}{z^{2-\sigma}} \Big|
    + t^2\lambda^2 \Big| \Big( \int_{\ray_a} - \int_{\arc_a} \Big) e^{-z+\frac{A^2}{z}} \,\frac{dz}{z^{3-\sigma}} \Big| \\
    &\lesssim \lambda \max(A^{2\sigma - 2}, A^{\sigma - \frac{3}{2}}) 
+ \lambda^2  \max(A^{2\sigma - 4},  A^{\sigma - \frac{5}{2}}) \\
    &\lesssim  \max(1,\lambda^{\frac{\sigma}{2} + \frac{3}{4}}).
  \end{align*}

\end{proof}

\begin{proof}[Proof of Theorem \ref{mainresult1}]
  Assume first that
  \begin{equation*}
    f\in 
    \begin{cases}
      \mathcal{L}^2 , &\sigma\leq \frac{1}{2} , \\
      \dom (L^{\frac{\sigma}{2} - \frac{1}{4}}) , &\sigma > \frac{1}{2} ,
    \end{cases}
  \end{equation*}
  and denote by $E$ the spectral measure of $L$.
  In order to see that the family
  \begin{equation*}
    u_{\varepsilon , R}(\cdot , t) = \frac{i^\sigma t^{2\sigma}}{4^\sigma \Gamma (\sigma)}
    \int_\varepsilon^R e^{-i\frac{t^2}{4s}} \, e^{-isL}f \,\frac{ds}{s^{1+\sigma}}
  \end{equation*}
  converges weakly in $\mathcal{L}^2$ as $\varepsilon\to 0$ and $R\to\infty$ note first that for any $h\in\mathcal{L}^2$,
  \begin{align*}
    \la u_{\varepsilon , R}(\cdot ,t) , h \ra
    &= \frac{i^\sigma t^{2\sigma}}{4^\sigma \Gamma (\sigma)}
    \int_\varepsilon^R e^{-i\frac{t^2}{4s}} \int_0^\infty e^{-is\lambda} \,dE_{f,h}(\lambda) \,\frac{ds}{s^{1+\sigma}} \\
    &= \frac{i^\sigma t^{2\sigma}}{4^\sigma \Gamma (\sigma)}
    \int_0^\infty \int_\varepsilon^R e^{-i\frac{t^2}{4s}} \, e^{-is\lambda} \,\frac{ds}{s^{1+\sigma}} \,dE_{f,h}(\lambda)  \\
    &= \frac{i^\sigma}{\Gamma (\sigma)} \int_0^\infty \int_{t^2/4R}^{t^2/4\varepsilon} e^{-is} e^{-i\frac{t^2}{4s}\lambda}
    \,\frac{ds}{s^{1-\sigma}} \,dE_{f,h}(\lambda) .
  \end{align*}
  Here, by Lemma \ref{firstestimate},
  \begin{equation*}
    \Big| \int_{t^2/4R}^{t^2/4\varepsilon} e^{-is} e^{-i\frac{t^2}{4s}\lambda} \,\frac{ds}{s^{1-\sigma}} \Big|
    \lesssim \begin{cases}
      1 , &\sigma\leq \frac{1}{2} , \\
      1 + \lambda^{\frac{\sigma}{2} - \frac{1}{4}} , &\sigma > \frac{1}{2} ,
    \end{cases}
  \end{equation*}
  from which the convergence follows. 
  
  
  Assume then that $f\in\dom (L^{\frac{\sigma}{2} + \frac{3}{4}})$. 
  By Lemma \ref{secondestimate} and Proposition \ref{lambdaeq}, we have
  \begin{align*}
    \Big( \partial_t^2 + \frac{1-2\sigma}{t} \partial_t \Big) \la u(\cdot , t) , h \ra
    &= \int_0^\infty \Big( \partial_t^2 + \frac{1-2\sigma}{t} \partial_t \Big) I_\sigma (\lambda , t) \,dE_{f,h}(\lambda) \\
    &=- \int_0^\infty \lambda I_\sigma (\lambda , t) \,dE_{f,h}(\lambda) \\
    &= -\la Lu(\cdot , t) , h \ra .
  \end{align*}
  
  Weak convergence to the initial data,
  \begin{equation*}
    \la u(\cdot , t), h \ra = \int_0^\infty I_\sigma (\lambda , t) \,dE_{f,h}(\lambda) \longrightarrow \la f , h \ra , \quad
    \textup{as} \quad t\to 0,
  \end{equation*}
  holds by Proposition \ref{lambdaeq} and Lemma \ref{firstestimate}.
  
  Moreover, taking into account that
  \begin{equation*}
    |I_{1-\sigma}(\lambda , t)| \lesssim
    \begin{cases}
      1 , &\sigma \geq \frac{1}{2},\\
      1 + \lambda^{\frac{1}{4} - \frac{\sigma}{2}}, &\sigma < \frac{1}{2} ,
    \end{cases}
  \end{equation*}
  we have, by Proposition \ref{lambdaeq}, that 
  \begin{align*}
    \partial_t^\sigma \la u(\cdot , t) , h \ra
    &= \int_0^\infty \partial_t^\sigma I_\sigma (\lambda , t) \,dE_{f,h}(\lambda) \\
    &= -i^{2\sigma} \frac{\Gamma (1-\sigma)}{\sigma 4^\sigma \Gamma (\sigma)} 
    \int_0^\infty \lambda^\sigma I_{1-\sigma}(\lambda , t) \,dE_{f,h}(\lambda) \\
    &\longrightarrow -i^{2\sigma} \frac{\Gamma (1-\sigma)}{\sigma 4^\sigma \Gamma (\sigma)}
    \la L^\sigma f , h \ra ,
  \end{align*}
  as $t\to 0$.
\end{proof}

\begin{remark}
  Note that the proof of Theorem \ref{mainresult1} entails the fixed-time norm estimate
  \begin{equation*}
    \| u(\cdot , t) \|_2
    \lesssim 
    \begin{cases}
      \| f \|_2 , &\sigma\leq \frac{1}{2} , \\
      \| f \|_2 + \| L^{\frac{\sigma}{2} - \frac{1}{4}}f \|_2 , &\sigma > \frac{1}{2} ;
    \end{cases}
  \end{equation*}
  cf. Theorem \ref{fixedtime}.
\end{remark}

\section{Kernel representations in the case of the Laplacian}
\label{sectionkernel}

Let $0 < \sigma < 1$ and consider the Neumann problem
\begin{equation}
\label{Neumannproblem}
  \begin{cases}
    \partial_t^2 u + \frac{1-2\sigma}{t}\partial_t u = \Delta u \\
    u(\cdot , 0) = 0 , \;\, \partial_t^\sigma u(\cdot , 0) = g ,
  \end{cases}
\end{equation}
for Schwartz initial data $g$ on $\R^d$. In this section we prove the following three results:

\begin{theorem}
\label{kernel123}
  Suppose that $d\in \{ 1,2,3 \}$ and let $0<\sigma < 1$ be such that $0 < \frac{d}{2} - \sigma < 1$. 
  A solution to the Neumann problem \eqref{Neumannproblem} with Schwartz initial data $g$ is given by
  \begin{equation*}
  u(x,t) = \frac{\sigma \sin((\frac{d}{2}-\sigma)\pi) \Gamma (\frac{d}{2} - \sigma)}{\pi^{\frac{d}{2}}\sin(\sigma\pi)\Gamma (1-\sigma)}
  \int_{B(x,t)} \frac{g(y)}{(t^2-|x-y|^2)^{\frac{d}{2} - \sigma}} \,dy .
  \end{equation*}
\end{theorem}

\begin{theorem}
\label{kernel345}
  Suppose that $d\in \{ 3,4,5 \}$ and let $0<\sigma < 1$ be such that $1 < \frac{d}{2} - \sigma < 2$. 
  A solution to the Neumann problem \eqref{Neumannproblem} with Schwartz initial data $g$ is given by
  \begin{equation*}
  u(x,t) = \frac{\sigma \sin((\frac{d}{2} - \sigma - 1)\pi) \Gamma (\frac{d}{2} - \sigma -1)}{2\pi^\frac{d}{2}\sin(\sigma\pi) \Gamma (1-\sigma)}
  \frac{1}{t^2} 
    \int_{B(x,t)} \frac{2(\sigma + 1)g(y) + (y-x)\cdot \nabla g(y)}{(t^2 - |x-y|^2)^{\frac{d}{2} - \sigma - 1}} \,dy .
  \end{equation*}
\end{theorem}

This can be continued to higher dimensions, but the formulas will be more complicated.

\begin{theorem}
\label{limitingcases}
  The solution $u_\sigma$ of the Neumann problem \eqref{Neumannproblem} with Schwartz initial data $g$ has the limit
  \begin{equation*}
      \lim_\sigma u_\sigma (x,t) = \frac{1}{c_dt} \int_{\partial B(x,t)} g(y)\,dS(y)
  \end{equation*}
  when
  \begin{itemize}
    \item $d=2$ and $\sigma \searrow 0$, in which case $c_d=2\pi$,
    \item $d=3$ and $\sigma \to \frac{1}{2}$, in which case $c_d = 4\pi$,   
    \item $d=4$ and $\sigma \nearrow 1$, in which case $c_d = 2\pi^2$.
  \end{itemize}
\end{theorem}

Before the proofs, we use the Schrödinger kernel to rewrite the solution formulas.
According to Corollary \ref{corollary}, a solution to \eqref{Neumannproblem} is given by
\begin{equation*}
  u(\cdot , t) = - \frac{\sigma 4^\sigma \Gamma (\sigma)}{\sin (\sigma \pi) \Gamma (1-\sigma)} \,
  \imag (U_t^\sigma ((-\Delta)^{-\sigma}g)),
\end{equation*}
where
\begin{equation*}
    U_t^\sigma ((-\Delta)^{-\sigma}g) = \frac{i^\sigma t^{2\sigma}}{4^\sigma \Gamma (\sigma)} 
    \int_0^\infty e^{-i\frac{t^2}{4s}} \, e^{is\Delta}(-\Delta)^{-\sigma} g \,\frac{ds}{s^{1+\sigma}}
    = \frac{i^\sigma}{\Gamma (\sigma)} 
    \int_0^\infty e^{-i\frac{t^2}{4s}} \, e^{is\Delta} g \,\frac{ds}{s^{1-\sigma}} ,
\end{equation*}
by change of variable. Hence $u(\cdot , t)$ is the imaginary part of the integral
\begin{equation}
\label{Neumannintegral}
  -\frac{\sigma 4^\sigma i^\sigma}{\sin (\sigma\pi) \Gamma (1-\sigma)} \int_0^\infty e^{-i\frac{t^2}{4s}} \, e^{is\Delta}g
  \,\frac{ds}{s^{1-\sigma}} .
\end{equation}

The Schr\"odinger group is given, for Schwartz functions $g$, by
\begin{equation*}
  e^{is\Delta}g(x) = \frac{1}{(4\pi is)^{\frac{d}{2}}} \int_{\R^d} e^{i\frac{|x-y|^2}{4s}} g(y) \,dy 
\end{equation*}
and in
\begin{equation*}
  \int_0^\infty e^{-i\frac{t^2}{4s}} \, e^{is\Delta}g(x) \,\frac{ds}{s^{1-\sigma}}
  = \frac{1}{(4\pi i)^{\frac{d}{2}}} \Big( \int_0^1 + \int_1^\infty \Big) 
  e^{-i\frac{t^2}{4s}} \int_{\R^d} e^{i\frac{|x-y|^2}{4s}} g(y) \, dy \,\frac{ds}{s^{1+\frac{d}{2}-\sigma}} 
  =: I_1(x) + I_2(x) ,
\end{equation*}
the left-hand side 
converges absolutely and uniformly in $x$ for all $d$ and $\sigma$ (restricting to $\sigma < \frac{1}{2}$ for $d=1$).
Indeed, 
\begin{equation*}
  |I_2(x)| \leq \| g \|_1 \int_1^\infty \,\frac{ds}{s^{1+\frac{d}{2}-\sigma}} < \infty ,
\end{equation*}
and 
\begin{equation*}
  \sup_{0 < s < 1} \| e^{is\Delta}g \|_\infty 
  \leq  \sup_{0 < s < 1} \| \widehat{e^{is\Delta}g} \|_1 = \| \widehat{g} \|_1 
\end{equation*}  
so that
\begin{equation*}
  |I_1(x)| \lesssim \| \widehat{g} \|_1 \int_0^1 \,\frac{ds}{s^{1-\sigma}} < \infty .
\end{equation*}

Now the expression \eqref{Neumannintegral} is
\begin{equation*}
  -\frac{\sigma 4^\sigma i^{\sigma - \frac{d}{2}}}{(4\pi)^{\frac{d}{2}} \sin (\sigma\pi ) \Gamma (1-\sigma)}
  \int_0^\infty e^{-i\frac{t^2}{4s}} \int_{\R^d} e^{i\frac{|x-y|^2}{4s}} g(y) \,dy \,\frac{ds}{s^{1+\frac{d}{2} - \sigma}} ,
\end{equation*}
and we define
\begin{equation*}
  V_\gamma g(x,t) = i^{-\gamma} \int_0^\infty e^{-i\frac{t^2}{4s}} \int_{\R^d} e^{i\frac{|x-y|^2}{4s}} g(y) \,dy 
  \,\frac{ds}{s^{1+\gamma}}
\end{equation*}
for a free parameter $\gamma > 0$. Thus
\begin{equation}
\label{solformula}
  u(x,t) = -\frac{\sigma 4^{\sigma - \frac{d}{2}}}{\pi^{\frac{d}{2}}\sin (\sigma\pi) \Gamma (1-\sigma)} \,
  \imag V_\gamma g(x,t) .
\end{equation}

Observe that
\begin{equation}
\label{recursion}
  \partial_t V_\gamma g(x,t) = \frac{t}{2} V_{\gamma +1}g(x,t)
\end{equation}
which after iteration gives, for any $n\geq 1$, that
\begin{equation*}
  V_{\gamma + n}g(x,t) = \Big( \frac{2}{t} \partial_t \Big)^n V_\gamma g(x,t) .
\end{equation*}

\subsection*{Cases $0 < \gamma < 1$ and $1 < \gamma < 2$}
When $0 < \gamma < 1$, the kernel can be computed from
\begin{equation*}
  V_\gamma g(x,t) = i^{-\gamma} \lim_{\varepsilon\to 0} \int_{\R^d} \int_\varepsilon^\infty
  e^{-i\frac{t^2 - |x-y|^2}{4s}} \,\frac{ds}{s^{1+\gamma}} g(y) \,dy .
\end{equation*}
In the inner integral here, we make a change of variable $\frac{t^2-|x-y|^2}{4s} = \pm r$, 
separating the cases $t > |x-y|$ and $t < |x-y|$.
Doing so we obtain, respectively,
\begin{equation*}
  \int_\varepsilon^\infty e^{-i\frac{t^2-|x-y|^2}{4s}} \,\frac{ds}{s^{1+\gamma}} 
  = \frac{4^\gamma}{(t^2 - |x-y|^2)^\gamma} \int_0^{\frac{t^2-|x-y|^2}{4\varepsilon}} e^{-ir} \,\frac{dr}{r^{1-\gamma}}
  \longrightarrow \frac{4^\gamma i^{-\gamma} \Gamma (\gamma)}{(t^2-|x-y|^2)^\gamma}
\end{equation*}
and
\begin{equation*}
  \int_\varepsilon^\infty e^{-i\frac{t^2-|x-y|^2}{4s}} \,\frac{ds}{s^{1+\gamma}} 
  = \frac{4^\gamma}{(|x-y|^2 - t^2)^\gamma} \int_0^{\frac{|x-y|^2 - t^2}{4\varepsilon}} e^{ir} \,\frac{dr}{r^{1-\gamma}}
  \longrightarrow \frac{4^\gamma i^\gamma \Gamma (\gamma)}{(|x-y|^2 - t^2)^\gamma} ,
\end{equation*}
where we also used Lemma \ref{complexgamma}. Consequently, we have by dominated convergence
\begin{equation}
\label{Vgamma}
  V_\gamma g(x,t) = 4^\gamma \Gamma (\gamma) \int_{\R^d}
  \frac{i^{-2\gamma}1_{B(x,t)}(y) + 1_{\R^d\setminus \overline{B(x,t)}}(y)}{|t^2 - |x-y|^2|^\gamma} g(y) \,dy .
\end{equation}

\begin{proof}[Proof of Theorem \ref{kernel123}]
  The formula for the solution $u$ follows from \eqref{solformula} and \eqref{Vgamma} with $\gamma = \frac{d}{2} - \sigma$.
\end{proof}

In order to calculate $\partial_t \imag V_\beta g(x,t)$ 
for $0 < \beta < 1$ we argue by change of variables $y' = \frac{y-x}{t}$:
\begin{align*}
  \partial_t \Big(\int_{B(x,t)} \frac{g(y)}{(t^2 - |x-y|^2)^\beta} \,dy \Big)
  &= \partial_t \Big( t^{d-2\beta} \int_{B(0,1)} \frac{g(x+ty')}{(1-|y'|^2)^\beta} \,dy' \Big) \\
  &= t^{d-2\beta -1} \int_{B(0,1)} \frac{(d-2\beta)g(x+ty') + ty'\cdot \nabla g(x+ty')}{(1-|y'|^2)^\beta} \,dy' \\
  &= \frac{1}{t} \int_{B(x,t)} \frac{(d-2\beta)g(y) + (y-x)\cdot \nabla g(y)}{(t^2 - |x-y|^2)^\beta} \,dy .
\end{align*}
Now \eqref{Vgamma} implies
\begin{equation}
\label{betaformula}
  \partial_t \imag V_\beta g(x,t) 
  = -\sin(\beta\pi)4^\beta\Gamma (\beta) 
  \frac{1}{t} \int_{B(x,t)} \frac{(d-2\beta)g(y) + (y-x)\cdot \nabla g(y)}{(t^2 - |x-y|^2)^\beta} \,dy .
\end{equation}

\begin{proof}[Proof of Theorem \ref{kernel345}]
  The formula for the solution follows by means of \eqref{recursion} 
  from \eqref{solformula} with $\gamma = \frac{d}{2} - \sigma$
  and \eqref{betaformula} with $\beta = \gamma - 1$, so that
  \begin{align*}
    u(x,t) 
    &= -\frac{\sigma 4^{\sigma - \frac{d}{2}}}{\pi^{\frac{d}{2}} \sin (\sigma\pi) \Gamma (1-\sigma)} \frac{2}{t} \partial_t \imag V_{\gamma - 1}g(x,t) \\
    &= 2c_{d,\sigma} \frac{1}{t^2} 
    \int_{B(x,t)} \frac{2(\sigma + 1)g(y) + (y-x)\cdot \nabla g(y)}{(t^2 - |x-y|^2)^{\frac{d}{2} - \sigma - 1}} \,dy ,
  \end{align*}
  where
  \begin{equation*}
    c_{d,\sigma} = \frac{\sigma 4^{\sigma - \frac{d}{2}} \sin ((\gamma - 1) \pi) 4^{\gamma - 1} \Gamma (\gamma - 1)}{\pi^\frac{d}{2} \sin (\sigma\pi) \Gamma (1-\sigma)}
    = \frac{\sigma \sin((\frac{d}{2} - \sigma - 1)\pi) \Gamma (\frac{d}{2} - \sigma -1)}{4\pi^\frac{d}{2}\sin(\sigma\pi) \Gamma (1-\sigma)} .
  \end{equation*}
\end{proof}

\subsection*{Limiting cases $\gamma \nearrow 1$ and $\gamma \searrow 1$}
For the case $\gamma \nearrow 1$, recall first that
\begin{equation*}
  \imag V_\gamma g(x,t) = -4^\gamma \Gamma (\gamma) \sin (\gamma\pi) \int_{B(x,t)} \frac{g(y)}{(t^2 - |x-y|^2)^\gamma} \,dy.
\end{equation*}
By the change of variable $z = y-x$ and writing $z' = \frac{tz}{|z|}$, the integral can be written
\begin{equation*}
   \int_{B(0,t)} \frac{g(x+z)}{(t^2 - |z|^2)^\gamma}\, dz =  \int_{B(0,t)} \frac{g(x+z)-g(x+z')}{(t^2 - |z|^2)^\gamma}\, dz
   + \int_{B(0,t)} \frac{g(x+z')}{(t^2 - |z|^2)^\gamma}\, dz =: I +II.
\end{equation*}
Here
\begin{equation*}
 |I| \leq  \|  \nabla g\|_\infty \int_{B(0,t)} \frac{t-|z|}{(t^2 - |z|^2)^\gamma}\, dz
\leq \| \nabla g\|_\infty \int_{B(0,t)} \frac{(t-|z|)^{1-\gamma}}{(t + |z|)^\gamma}\, dz \lesssim \| \nabla g\|_\infty
\end{equation*}
and, with polar coordinates,
\begin{equation*}
 II = \int_{|\omega|=1} g(x+t\omega)\,dS(\omega)\, 
\int_0^t \frac{r^{d-1}}{(t^2 - r^2)^\gamma}\, dr.
\end{equation*}
After the transformation $\rho = r^2/t^2$, the inner integral here will be
\begin{equation*}
\frac{t^{d-2\gamma}}2 \int_0^1 \frac{\rho^{\frac d2-1}}{(1-\rho)^\gamma}\, d\rho
= \frac{t^{d-2\gamma}}2\, 
\frac{\Gamma(\frac d2) \Gamma(1-\gamma)}{  \Gamma(\frac d2+1-\gamma)}.
\end{equation*}
Since $\Gamma(1-\gamma) = \frac1{1-\gamma}+O(1)$ and 
$\sin(\gamma\pi) = (1-\gamma)\pi +O((1-\gamma)^2)$ as $\gamma \to 1$, we conclude that
\begin{equation*}
\label{gammalimit}
  \imag V_\gamma g(x,t) \to  
-4\pi \frac{t^{d-2}}2\,\int_{|\omega|=1} g(x+t\omega)\,dS(\omega) 
= -\frac{2\pi}t \int_{\partial B(x,t)} g(y) \,dS(y),
\end{equation*}
 as $\gamma \nearrow 1$.

Consider then the case $\gamma \searrow 1$ and begin by recalling that
\begin{align*}
  \imag V_\gamma g(x,t) &= \frac{2}{t} \partial_t \imag V_{\gamma - 1} g(x,t) \\
  &= -\frac{2}{t^2} 4^{\gamma - 1} \Gamma (\gamma - 1) \sin ((\gamma - 1)\pi)
  \int_{B(0,t)} \frac{(d-2(\gamma - 1))g(x+y) + y\cdot \nabla g(x+y)}{(t^2 - |y|^2)^{\gamma - 1}} \,dy .
\end{align*}
Since $\Gamma (\gamma - 1) \sin ((\gamma - 1)\pi) \longrightarrow \pi$ as $\gamma\to 1$,
\begin{equation*}
  \lim_{\gamma\searrow 1} \imag V_\gamma g(x,t)
  = -\frac{2\pi}{t^2} \int_{B(0,t)} \Big( dg(x+y) + y\cdot\nabla g(x+y) \Big) \,dy .
\end{equation*}
Noting that
\begin{equation*}
  d = \Delta h \quad \textup{and} \quad y = \nabla h \quad \textup{for} \quad h(y) = \frac{|y|^2}{2},
\end{equation*}
we make use of Green's formula
\begin{equation*}
  \int_{B(0,t)} \Big( (\Delta h) g + \nabla h \cdot \nabla g \Big) \,dy 
  = \int_{\partial B(0,t)} (\nabla h \cdot \frac{y}{t}) g \,dS(y)
\end{equation*}
to see that also
\begin{equation}
\label{gammalimit2}
  \lim_{\gamma\searrow 1} \imag V_\gamma g(x,t) = -\frac{2\pi}{t} \int_{\partial B(x,t)} g(y) \,dS(y) .
\end{equation}

\begin{proof}[Proof of Theorem \ref{limitingcases}]
  Note first that $\gamma = \frac{d}{2} - \sigma \to 1$ is possible in dimensions $d=2$, $3$ and $4$ when $\sigma \searrow 0$,
  $\sigma \to \frac{1}{2}$ and $\sigma \nearrow 1$, respectively.
  Since
  \begin{equation*}
    u_\sigma (x,t) = -\frac{\sigma 4^{\sigma - \frac{d}{2}}}{\pi^{\frac{d}{2}}\sin (\sigma\pi) \Gamma (1-\sigma)} \,
  \imag V_{\frac{d}{2} - \sigma} g(x,t) ,
  \end{equation*}
  together with the limits \eqref{gammalimit} and \eqref{gammalimit2} it suffices to note that
  \begin{equation*}
    \frac{\sigma 4^{\sigma - \frac{d}{2}}}{\pi^{\frac{d}{2}}\sin (\sigma\pi) \Gamma (1-\sigma)} 
    \longrightarrow 
      \begin{cases}
        \frac{1}{4\pi^2} , \quad &\textup{when} \quad d=2 \quad \textup{and} \quad \sigma \searrow 0 , \\
        \frac{1}{8\pi^2} , \quad &\textup{when} \quad d=3 \quad \textup{and} \quad \sigma \to \frac{1}{2} , \\
        \frac{1}{4\pi^3} , \quad &\textup{when} \quad d=4 \quad \textup{and} \quad \sigma \nearrow 1 .
      \end{cases}
  \end{equation*}
\end{proof}

\begin{remark}
  The one-dimensional problem
  \begin{equation}
    \begin{cases}
      \partial_t^2 u + \frac{1-2\sigma}{t}\partial_t u = \partial_x^2 u \\
      u(\cdot , 0)=0 , \;\, \partial_t^\sigma u (\cdot , 0) = g
    \end{cases}
  \end{equation}
  is not covered by Theorem \ref{kernel123} when $\sigma > \frac{1}{2}$. A solution formula can be derived by the method of
  descent by viewing $u$ and $g$ with an additional spatial variable and using Theorem \ref{kernel123} for $d=2$.
  Doing so we obtain
  \begin{equation*}
    u(x,t) = \frac{\sigma}{\pi} t^{2\sigma} \iint_{y_1^2 + y_2^2 < 1} \frac{g(x+ty_1)}{(1-y_1^2 - y_2^2)^{1-\sigma}} dy_1dy_2 .
  \end{equation*}
  Here the double integral can be calculated as follows:
  \begin{equation*}
    \int_{-1}^1 \frac{g(x+ty_1)}{(1-y_1^2)^{1-\sigma}} \int_{y_2^2 < 1-y_1^2} \Big( 1 - \frac{y_2^2}{1-y_1^2} \Big)^{\sigma - 1}
    \,dy_2\,dy_1 = \int_{-1}^1 \frac{g(x+ty_1)}{(1-y_1^2)^{\frac{1}{2}-\sigma}} \int_{-1}^1 (1-s^2)^{\sigma - 1} \,ds \,dy_1 ,
  \end{equation*}
  where
  \begin{equation*}
    \int_{-1}^1 (1-s^2)^{\sigma - 1} \,ds 
    = \int_0^1 (1-r)^{\sigma - 1} r^{-\frac{1}{2}} \,dr    
    = \frac{\Gamma (\sigma) \Gamma (\frac{1}{2})}{\Gamma (\sigma + \frac{1}{2})} .
  \end{equation*}
  Therefore
  \begin{equation*}
    u(x,t) = \frac{\sigma \Gamma (\sigma)}{\pi^{\frac{1}{2}} \Gamma (\sigma + \frac{1}{2})}
    \int_{x-t}^{x+t} \frac{g(y)}{ (t^2 - |x-y|^2)^{\frac{1}{2} - \sigma}} \,dy ,
  \end{equation*}
  which converges to
  \begin{equation*}
    \frac{1}{2} \int_{x-t}^{x+t} g(y) \,dy ,
  \end{equation*}
  as $\sigma \searrow \frac{1}{2}$. This coincides with the limit as $\sigma\nearrow \frac{1}{2}$.
\end{remark}

\begin{remark}
  In a similar vein one may consider the modified Klein--Gordon equation arising from $L=-\Delta + m^2$ with $m>0$.
  The solution to \eqref{Neumannproblem} is then given by
  \begin{equation*}
    u(x,t) = -\frac{\sigma 4^{\sigma - \frac{d}{2}}}{\pi^{\frac{d}{2}} \sin (\sigma\pi) \Gamma (1-\sigma)} \,
    \imag \Big( i^{\sigma - \frac{d}{2}} \int_0^\infty e^{-i\frac{t^2}{4s}} \int_{\R^d}
    e^{i\frac{|x-y|^2}{4s} -im^2s} g(y) \,dy \,\frac{ds}{s^{1+\frac{d}{2} - \sigma}} \Big) , 
  \end{equation*}
  where the quantity inside $\imag$ coincides with
  \begin{equation*}
    V_\gamma^m g(x,t) = i^{-\gamma} \lim_{\varepsilon\to 0} \int_{\R^d}
    \int_\varepsilon^\infty e^{i\frac{|x-y|^2 - t^2}{4s} -im^2s} \,\frac{ds}{s^{1+\gamma}} \, g(y) \,dy ,
  \end{equation*}
  with $\gamma = \frac{d}{2} - \sigma$. Assuming that $0 < \gamma < 1$, we may use the formula (cf. Lemma \ref{complexgamma})
  \begin{equation*}
    \int_0^\infty e^{-\frac{A}{is} - im^2s} \,\frac{ds}{s^{1+\gamma}}
    = i\ ^\gamma \int_0^\infty e^{-\frac{A}{s} - m^2s} \,\frac{ds}{s^{1+\gamma}}
  \end{equation*}
  when $A = \frac{|x-y|^2 - t^2}{4} > 0$ to deduce finite speed of propagation:
  \begin{equation*}
    u(x,t) = -\frac{\sigma 4^{\sigma - \frac{d}{2}}}{\pi^{\frac{d}{2}} \sin (\sigma\pi) \Gamma (1-\sigma)} \,
    \imag \Big( i^{\sigma - \frac{d}{2}} \lim_{\varepsilon\to 0} \int_{B(x,t)}
    \int_\varepsilon^\infty e^{i\frac{|x-y|^2 - t^2}{4s} -im^2s} \,\frac{ds}{s^{1+\frac{d}{2} - \sigma}} \, g(y) \,dy \Big) . 
  \end{equation*}
\end{remark}

\section{A Bessel function approach}
\label{sectionBessel}

In this section we study the wave extension problem for the Laplacian $L=-\Delta$ on $\R^d$ by means of Bessel functions  
and obtain some elementary estimates for the solutions. In order to simplify the presentation, we make stronger decay and
regularity assumptions than necessary.

For a given $\sigma\in (0,1)$ we consider the problem
\begin{equation}
\label{sigmaWE}
\begin{cases}
  \partial_t^2 u + \frac{1-2\sigma}{t}\partial_t u = \Delta u \\
  u(\cdot , 0)=f, \;\, \partial_t^\sigma u(\cdot , 0) = g ,
\end{cases}
\end{equation}
where $f$ and $g$ are Schwartz functions whose Fourier transforms vanish in a neighbourhood of 
the origin, and $\partial_t^\sigma = \frac{1}{2\sigma} t^{1-2\sigma} \partial_t$.

We interpret the problem as an evolution equation in $\mathcal{L}^2(\R^d)$, and call $u$ a \emph{Schwartz solution} if
the map $(0,\infty) \to \mathcal{L}^2(\R^d) : t\mapsto u(\cdot , t)$ is $C^2$ (in the norm topology), 
and $u(\cdot , t)$, $\partial_t u(\cdot ,t)$,
and $\partial_t^2 u(\cdot ,t)$ are Schwartz functions for every $t>0$.
The Fourier transform of such a solution $u$ 
then satisfies for every $\xi\in\R^d$ the equation
\begin{equation}
\label{fourier}
  \partial_t^2 \widehat{u}(\xi,t) + \frac{1-2\sigma}{t}\partial_t \widehat{u}(\xi,t) = -|\xi|^2 \widehat{u}(\xi,t) ,
  \quad t>0 .
\end{equation}

Writing $\widehat{u}(\xi , t) = (t|\xi|)^\sigma w(t|\xi|)$ we get for $w$
\begin{equation*}
  r^2w''(r) + rw'(r) + (r^2 - \sigma^2) w(r) = 0 ,
\end{equation*}
which is Bessel's differential equation. Thus two linearly independent solutions to \eqref{fourier} are
$\widehat{u}(\xi , t) = (t|\xi|)^\sigma J_{\pm\sigma}(t|\xi|)$; see \cite[Sections 5.1 and 5.3]{LEBEDEV}.

Let us then find the coefficients that give convergence to the correct initial values. The Bessel functions satisfy for
$0 < \sigma < 1$ and $r>0$
\begin{equation*}
  J_\sigma (r) = \frac{1}{\Gamma (1 + \sigma)} \Big( \frac{r}{2} \Big)^\sigma + O(r^{\sigma + 2})
  \quad \textup{and} \quad
  J_{-\sigma} (r) = \frac{1}{\Gamma (1 - \sigma)} \Big( \frac{r}{2} \Big)^{-\sigma} + O(r^{-\sigma + 2})
\end{equation*}
as $r\searrow 0$, as seen from the power series expansion. Further,
\begin{equation*}
  \frac{d}{dr} (r^\sigma J_{\pm\sigma}(r)) = \pm r^\sigma J_{\pm (\sigma - 1)}(r) ;
\end{equation*}
see \cite[Formula (5.3.5)]{LEBEDEV}. This implies
\begin{equation*}
  (t|\xi|)^\sigma J_\sigma (t|\xi|) \longrightarrow 0 \quad \textup{and} 
  \quad (t|\xi|)^\sigma J_{-\sigma}(t|\xi|) \longrightarrow \frac{2^\sigma}{\Gamma (1-\sigma)}
\end{equation*}
as $t\to 0$. Moreover,
\begin{equation*}
  \partial_t^\sigma \Big( (t|\xi|)^\sigma J_\sigma (t|\xi|) \Big) 
  \longrightarrow \frac{|\xi|^{2\sigma}}{\sigma 2^\sigma \Gamma (\sigma)} \quad \textup{and} 
  \quad \partial_t^\sigma \Big( (t|\xi|)^\sigma J_{-\sigma}(t|\xi|) \Big) \longrightarrow 0
\end{equation*}
as $t\to 0$.

Therefore the solution to \eqref{fourier} given by
\begin{equation}
\label{sol}
  \widehat{u}(\xi,t) = \frac{\Gamma (1-\sigma)}{2^\sigma}(t|\xi|)^\sigma J_{-\sigma}(t|\xi|) \widehat{f}(\xi)
  + \sigma 2^\sigma \Gamma (\sigma) (t|\xi|)^\sigma J_\sigma (t|\xi|) \, |\xi|^{-2\sigma}\widehat{g}(\xi)
\end{equation}
satisfies
\begin{equation*}
  \widehat{u}(\xi,t) \longrightarrow \widehat{f}(\xi) \quad \textup{and} \quad 
  \partial_t^\sigma\widehat{u}(\xi,t) \longrightarrow \widehat{g}(\xi) \quad \textup{as}
  \quad t\to 0.
\end{equation*}

Since $\widehat{f}$ and $\widehat{g}$ were assumed to vanish in a neighbourhood of the origin, the next result
follows immediately by dominated convergence:

\begin{theorem}
\label{Besselsol}
Let $0 < \sigma < 1$. The unique Schwartz solution $u$ to the problem
\begin{equation*}
\begin{cases}
  \partial_t^2 u + \frac{1-2\sigma}{t}\partial_t u = \Delta u \\
  u(\cdot , 0)=f, \;\, \partial_t^\sigma u(\cdot , 0) = g ,
\end{cases}
\end{equation*}
with Schwartz initial data $f$ and $g$ whose Fourier transforms vanish in a neighbourhood of the origin, is given by
\begin{equation*}
  u(\cdot , t) = \frac{\Gamma (1-\sigma)}{2^\sigma} (t\sqrt{-\Delta})^\sigma J_{-\sigma}(t\sqrt{-\Delta}) f 
  + \sigma 2^\sigma \Gamma (\sigma) (t\sqrt{-\Delta})^\sigma J_\sigma (t\sqrt{-\Delta}) (-\Delta)^{-\sigma} g,
\end{equation*}
which convergences to the initial data in $\mathcal{L}^2(\R^d)$.
\end{theorem}

\begin{remark}
  The same method can be used to prove uniqueness of solutions also in other settings, say for
  non-negative self-adjoint operators $L$ with discrete spectrum (such as the Hermite operator $L=-\Delta + |x|^2$ on $\R^d$),
  by replacing $|\xi| = \sqrt{\lambda}$.
  In such a case the solution $u$ of
  \begin{equation*}
    \begin{cases}
      \partial_t^2 u + \frac{1-2\sigma}{t}\partial_t u = -L u \\
      u(\cdot , 0)=f, \;\, \partial_t^\sigma u(\cdot , 0) = g ,
    \end{cases}
  \end{equation*}
  is unique under the assumption that $u$, $f$ and $g$ all consist of a finite number of eigenfunctions.
\end{remark}

\subsection*{Fixed-time norm estimates}
For any order of smoothness $s\in\R$ we define the Sobolev norm of a Schwartz function $f$ on $\R^d$ by
\begin{equation*}
  \| f \|_{2,s} = \Big( \int_{\R^d} |\la \xi \ra^s \widehat{f}(\xi)|^2 \,d\xi \Big)^{\frac{1}{2}} ,
\end{equation*}
where $\la \xi \ra = (1+|\xi|^2)^{\frac{1}{2}}$. As a consequence of Theorem \ref{Besselsol}, we derive the following
fixed-time norm estimates for solutions $u$ in the spirit of \cite[Theorem 2]{STRICONV}:  

\begin{theorem}
\label{fixedtime}
  Let $s\geq 0$ and let $u$, $f$ and $g$ be as in Theorem \ref{Besselsol}.
  \begin{itemize}
  \item If $\sigma \leq \frac{1}{2}$ then 
  $\| u(\cdot , t) \|_{2,s} \lesssim \| f \|_{2,s} + t^{2\sigma} \| g \|_2 + \| g \|_{2, s-2\sigma}$
  for all $t>0$.
  \item If $\sigma \geq \frac{1}{2}$ then 
  $\| u(\cdot , t) \|_{2,s} \lesssim (1+t)^{\sigma - \frac{1}{2}} \| f \|_{2,s+\sigma - \frac{1}{2}} + t^{2\sigma} \| g \|_2 +
  t^{\sigma - \frac{1}{2}} \| g \|_{2,s-\sigma - \frac{1}{2}}$ for all $t>0$.
  \end{itemize}
  \begin{proof}
    Let $\chi$ be a frequency cut-off with $1_{B(0,1)} \leq \chi \leq 1_{B(0,2)}$
    and recall the standard estimates for Bessel functions: 
    \begin{equation}
      \label{allr}      
        |J_{\pm\sigma} (r)| \lesssim
      \begin{cases}
        r^{\pm\sigma}, &0 < r < 1, \\
        r^{-\frac{1}{2}}, &r \geq 1 .
      \end{cases}
    \end{equation}    
    
    Consider first the case $0 < \sigma \leq \frac{1}{2}$. By \eqref{allr} we have
    \begin{equation*}
      (t|\xi|)^\sigma |J_{-\sigma}(t|\xi|)| \lesssim 1
    \end{equation*}
    and therefore
    \begin{equation*}
      \| (t\sqrt{-\Delta})^\sigma J_{-\sigma}(t\sqrt{-\Delta}) f \|_{2,s} \lesssim \| f \|_{2,s} .
    \end{equation*}
    On the other hand, according to \eqref{allr} we have
    $(t|\xi|)^\sigma |J_\sigma (t|\xi|)| \lesssim \min \{ (t|\xi|)^{2\sigma} , 1 \}$, which implies that
    \begin{align*}
      \la \xi \ra^s (t|\xi|)^\sigma |J_\sigma (t|\xi|)| \, |\xi|^{-2\sigma}
      &\lesssim \chi (\xi) \la \xi \ra^s (t|\xi|)^{2\sigma} |\xi|^{-2\sigma} + (1-\chi (\xi)) \la \xi \ra^s |\xi|^{-2\sigma}\\
      &\lesssim t^{2\sigma} \chi (\xi) + (1-\chi (\xi)) \la \xi \ra^{s-2\sigma} ,
    \end{align*}
    and thus
    \begin{equation*}
      \| (t\sqrt{-\Delta})^\sigma J_{\sigma}(t\sqrt{-\Delta}) (-\Delta)^{-\sigma}g \|_{2,s} 
      \lesssim t^{2\sigma} \| g \|_2 + \| g \|_{2,s-2\sigma} .
    \end{equation*}
    
    Assume then that $\frac{1}{2} \leq \sigma < 1$. By \eqref{allr} we have
    $(t|\xi|)^\sigma |J_{-\sigma}(t|\xi|)| \lesssim \max \{ 1 , (t|\xi|)^{\sigma - \frac{1}{2}} \}$ so that
    \begin{equation*}
      \la \xi \ra^s (t|\xi|)^\sigma |J_{-\sigma}(t|\xi|)|
      \lesssim \chi (t\xi) \la \xi \ra^s + (1-\chi (t\xi)) \la \xi \ra^s (t|\xi|)^{\sigma - \frac{1}{2}} ,
    \end{equation*}
    where $\la \xi \ra^s (t|\xi|)^{\sigma - \frac{1}{2}} \leq t^{\sigma - \frac{1}{2}} \la \xi \ra^{s+\sigma - \frac{1}{2}}$.
    Consequently,
    \begin{equation*}
      \| (t\sqrt{-\Delta})^\sigma J_{-\sigma}(t\sqrt{-\Delta}) f \|_{2,s}
      \lesssim \| f \|_{2,s} + t^{\sigma - \frac{1}{2}} \| f \|_{2,s+\sigma - \frac{1}{2}}
      \lesssim (1+t)^{\sigma - \frac{1}{2}} \| f \|_{2,s+\sigma - \frac{1}{2}} .
    \end{equation*}
    Moreover, from \eqref{allr} we see that
    $(t|\xi|)^\sigma |J_\sigma (t|\xi|)| \lesssim \min \{ (t|\xi|)^{2\sigma} , (t|\xi|)^{\sigma - \frac{1}{2}} \}$ and so
    \begin{align*}
      \la \xi \ra^s (t|\xi|)^\sigma |J_\sigma (t|\xi|)| \, |\xi|^{-2\sigma}
      &\lesssim \chi (\xi ) \la \xi \ra^s (t|\xi|)^{2\sigma} |\xi|^{-2\sigma}
      + (1-\chi(\xi)) (t|\xi|)^{\sigma - \frac{1}{2}} \la \xi \ra^s |\xi|^{-2\sigma} \\
      &\lesssim t^{2\sigma} \chi(\xi) + (1-\chi(\xi)) t^{\sigma - \frac{1}{2}} \la \xi \ra^{s-\sigma - \frac{1}{2}} ,
    \end{align*}
    from which infer that
    \begin{equation*}
      \| (t\sqrt{-\Delta})^\sigma J_\sigma (t\sqrt{-\Delta}) (-\Delta)^{-\sigma} g \|_{2,s}
      \lesssim t^{2\sigma} \| g \|_2 + t^{\sigma - \frac{1}{2}} \| g \|_{2,s-\sigma - \frac{1}{2}} .
    \end{equation*}
  \end{proof}
\end{theorem}

\subsection*{An integral formula}

We finish by deriving the oscillatory subordination formula of Theorem \ref{mainresult1} from a classical integral representation
for modified Bessel functions 
\begin{equation*}
  K_\sigma (z) = \frac{\pi}{2} \frac{I_{-\sigma}(z) - I_\sigma (z)}{\sin (\sigma \pi)} , \quad |\arg z| < \pi ,
\end{equation*}
(see \cite[Section 5.7]{LEBEDEV}).
Noting that
\begin{equation*}
  K_\sigma (ir) = \frac{\pi}{2} \frac{i^{-\sigma} J_{-\sigma}(r) - i^\sigma J_\sigma (r)}{\sin (\sigma \pi)} , \quad r>0 ,
\end{equation*}
we see that the solution $u$ in Theorem \ref{Besselsol} for real initial data can be written as
\begin{align*}
  u(\cdot , t) &= \frac{2}{\pi} \frac{\Gamma (1-\sigma)}{2^\sigma} \, \real \Big( i^{1-2\sigma} (it\sqrt{-\Delta})^\sigma
  K_\sigma (it\sqrt{-\Delta})f \Big) \\
  &\mbox{} - \frac{2}{\pi} \sigma 2^\sigma \Gamma (\sigma) \, \imag \Big( (it\sqrt{-\Delta})^\sigma
  K_\sigma (it\sqrt{-\Delta}) (-\Delta)^{-\sigma}g \Big) .
\end{align*}


\begin{prop}
\label{Besselintform}
Let $0<\sigma < 1$. We have
\begin{equation*}
  (it\sqrt{-\Delta})^\sigma K_\sigma (it\sqrt{-\Delta}) f = \frac{i^\sigma t^{2\sigma}}{2^{1+\sigma}}
  \int_0^\infty e^{-i\frac{t^2}{4s}} \, e^{is\Delta} f \,\frac{ds}{s^{1+\sigma}} 
\end{equation*}
for Schwartz functions $f$.
\end{prop}
\begin{proof}
We first show that
\begin{equation*}
  K_\sigma (ir) = \frac{r^\sigma}{2^{1+\sigma}}
  \int_0^\infty e^{-i\frac{r^2}{4s}} \, e^{-is} \,\frac{ds}{s^{1+\sigma}} , \quad r>0 .
\end{equation*}

By \cite[Eq. (5.10.25)]{LEBEDEV} we have the integral representation
\begin{equation*}
  K_\sigma (z) = \frac{z^\sigma}{2^{1+\sigma}} \int_0^\infty e^{-\frac{z^2}{4s}} \, e^{-s} \,\frac{ds}{s^{1+\sigma}}, 
  \quad |\arg z | < \frac{\pi}{4} .
\end{equation*}
Note that when $|\arg z| > \frac{\pi}{4}$, this integral diverges at zero because $\real \, e^{-\frac{z^2}{4s}} \gtrsim 1$.

If we introduce the bianalytic function
\begin{equation}
\label{bian}
  F(z,w) = e^{-\frac{z^2}{4w}} \, e^{-w} \, w^{-1-\sigma} , \quad |\arg z| < \frac{\pi}{2} , \: |\arg w | < \pi ,
\end{equation}
the integral above can be written $\int_0^\infty F(z,s) \, ds$.

Integrating instead $w\mapsto F(z,w)$ along the imaginary half-axis, we obtain the analytic function
\begin{equation*}
  z \mapsto i^{-\sigma} \int_0^\infty e^{i\frac{z^2}{4s}} \, e^{-is} \,\frac{ds}{s^{1+\sigma}} , \quad 0 < \arg z < \frac{\pi}{2} .
\end{equation*}
Indeed, for such $z$ the integral is absolutely convergent since 
$|e^{i\frac{z^2}{4s}}| = e^{-c\frac{|z|^2}{4s}}$, where $c = \sin (2\arg z) > 0$.
We claim that
\begin{equation}
\label{imagpath}
  \int_0^\infty F(z,s) \,ds = i\int_0^\infty F(z,is) \,ds , \quad 0 < \arg z < \frac{\pi}{4} .
\end{equation}

To see this, we let $\varepsilon >0$ and $R < \infty$, and use Cauchy's integral theorem to write
\begin{equation*}
  \int_\varepsilon^R F(z,s) \,ds = \Big( \int_{\arc_\varepsilon} - \int_{\arc_R} \Big) F(z,w) \,dw 
  + i \int_\varepsilon^R F(z,is) \,ds ,
\end{equation*}
where $\arc_\varepsilon$ and $\arc_R$ are the paths
\begin{equation*}
  \theta \mapsto \varepsilon e^{i\theta} \quad \textup{and} \quad \theta\mapsto Re^{i\theta} , 
  \quad 0 \leq \theta \leq \frac{\pi}{2} .
\end{equation*}
Then
\begin{equation*}
  \int_{\arc_\varepsilon} F(z,w) \,dw = i \int_0^{\frac{\pi}{2}} e^{-\frac{z^2}{4\varepsilon}e^{-i\theta}}
  e^{-\varepsilon e^{i\theta}} \varepsilon^{-\sigma} e^{-i\sigma\theta} \,d\theta ,
\end{equation*}
and with $\phi = \arg z \in (0,\frac{\pi}{4})$ we have
\begin{equation*}
  | e^{-\frac{z^2}{4\varepsilon}e^{-i\theta}} | = e^{-|z|^2\cos (2\phi - \theta) /4\varepsilon}
  = e^{-|z|^2\sin (\frac{\pi}{2} - 2\phi + \theta) /4\varepsilon} .
\end{equation*}
Since $\sin \alpha > c\min (\alpha , \pi - \alpha )$ for $0 < \alpha < \pi$, it follows that
\begin{equation*}
  \sin \Big( \frac{\pi}{2} - 2\phi + \theta \Big) > c\min (\theta , \pi / 2 - \theta ) .
\end{equation*}
As a consequence,
\begin{equation*}
  \Big| \int_{\arc_\varepsilon} F(z,w) \,dw \Big|
  \lesssim \varepsilon^{-\sigma} \int_0^{\frac{\pi}{2}} 
  e^{-\frac{c|z|^2\min (\theta , \pi / 2 - \theta )}{4\varepsilon}} \,d\theta
  \lesssim \varepsilon^{1-\sigma} \longrightarrow 0 , \quad \textup{as} \quad \varepsilon \to 0.
\end{equation*}
Likewise,
\begin{equation*}
  \int_{\arc_R} F(z,w) \,dw = i \int_0^{\frac{\pi}{2}} e^{-\frac{z^2}{4R}e^{-i\theta}}
  e^{-R e^{i\theta}} R^{-\sigma} e^{-i\sigma\theta} \,d\theta \longrightarrow 0 , \quad \textup{as} \quad R \to \infty ,
\end{equation*}
and \eqref{imagpath} follows.

Now since the functions
\begin{equation*}
  z\mapsto \frac{z^\sigma}{2^{1+\sigma}} i\int_0^\infty F(z,is) \,ds
\end{equation*}
and $K_\sigma (z)$ are both analytic for $0 < \arg z < \frac{\pi}{2}$ and agree for $0 < \arg z < \frac{\pi}{4}$
because of \eqref{imagpath}, they must coincide, so that
\begin{equation*}
  K_\sigma (z) = \frac{i^{-\sigma}z^\sigma}{2^{1+\sigma}} \int_0^\infty e^{i\frac{z^2}{4s}} \, e^{-is} 
  \,\frac{ds}{s^{1+\sigma}} , \quad 0 < \arg z < \frac{\pi}{2} .
\end{equation*}

Finally,
\begin{equation*}
  K_\sigma (ir) = \lim_{\delta\to 0} K_\sigma (ir+\delta) 
  = \frac{r^\sigma}{2^{1+\sigma}} \lim_{\delta\to 0} \int_0^\infty e^{i\frac{(ir+\delta)^2}{4s}} e^{-is} 
  \,\frac{ds}{s^{1+\sigma}}
  = \frac{r^\sigma}{2^{1+\sigma}} \lim_{\varepsilon , R} \int_\varepsilon^R 
  e^{-i\frac{r^2}{4s}} \, e^{-is} \,\frac{ds}{s^{1+\sigma}} ,
\end{equation*}
where the swapping of the limits in the last equality is justified since the arguments above show that
the limit in $\varepsilon$ and $R$ is uniform in $\delta$.

The proposition now follows by Fourier transform from the identity
\begin{equation*}
  (it|\xi|)^\sigma K_\sigma (it|\xi|)
  = \frac{i^\sigma (t|\xi|)^{2\sigma}}{2^{1+\sigma}} \lim_{\varepsilon , R} \int_\varepsilon^R e^{-i\frac{t^2}{4s}|\xi|^2}  e^{-is} \,\frac{ds}{s^{1+\sigma}}
  = \frac{i^\sigma t^{2\sigma}}{2^{1+\sigma}} \lim_{\varepsilon , R} \int_\varepsilon^R e^{-i\frac{t^2}{4s}} \, e^{-is|\xi|^2} \,\frac{ds}{s^{1+\sigma}} .
\end{equation*}
Indeed, after a change of variables $s' = \frac{t^2}{4s}$ the last integral is amenable to Lemma
\ref{firstestimate} which gives
\begin{equation*}
  \Big| \int_\varepsilon^R e^{-i\frac{t^2}{4s}} \, e^{-is|\xi|^2} \,\frac{ds}{s^{1+\sigma}} \Big| 
  \lesssim
  \begin{cases}
      1 , &\sigma\leq \frac{1}{2} , \\
      1 + |\xi|^{\sigma - \frac{1}{2}} , &\sigma > \frac{1}{2} .
  \end{cases}
\end{equation*}
The functions
\begin{equation*}
  \xi \mapsto \frac{i^\sigma t^{2\sigma}}{2^{1+\sigma}} \int_\varepsilon^R e^{-i\frac{t^2}{4s}} \, e^{-is|\xi|^2} 
  \widehat{f}(\xi) \,\frac{ds}{s^{1+\sigma}} 
\end{equation*}
therefore converge in $\mathcal{L}^2(\R^d)$, as $\varepsilon\to 0$ and $R\to\infty$, which justifies the exchange of limits with
the Fourier transform.
\end{proof}

\begin{remark}
  It is clear that Proposition \ref{Besselintform} holds for any non-negative self-adjoint operator $L$ in 
  place of $-\Delta$ when $f$ is suitably chosen (cf. Theorem \ref{mainresult1}). 
  Moreover, recalling that
  \begin{equation*}
    \sqrt{ir} K_{1/2} (ir) = \sqrt{\frac{\pi}{2}} e^{-ir} , \quad r>0,
  \end{equation*}
  we see that in the classical case $\sigma = \frac{1}{2}$ the formula reads
  \begin{equation*}
  e^{-it\sqrt{L}} f = \frac{\sqrt{i}t}{2\sqrt{\pi}}
  \int_0^\infty e^{-i\frac{t^2}{4s}} \, e^{-isL} f \,\frac{ds}{s^{3/2}} . 
\end{equation*}
\end{remark}

\bibliographystyle{alpha}
\bibliography{references}

\end{document}